\documentclass[12pt,reqno,twoside]{amsart}%

\usepackage{amsmath,amsthm,amscd,amsthm,upref,indentfirst}
\usepackage{amsfonts,mathrsfs}
\usepackage{amssymb,amsbsy,bm}
\usepackage{graphicx}

\usepackage{amsaddr}

\usepackage{soul}

\usepackage[us,long,24hr]{datetime}

\usepackage{mathabx}

\usepackage[usenames,dvipsnames]{color}

\theoremstyle{plain}
\newtheorem{theorem}{Theorem}

\newtheorem{proposition}[theorem]{Proposition}

\theoremstyle{definition}
\newtheorem{definition}[theorem]{Definition}

\theoremstyle{remark}
\newtheorem{remark}[theorem]{Remark}

\newtheorem{example}[theorem]{Example}
\numberwithin{equation}{section}
\numberwithin{theorem}{section}

\normalfont\upshape
\usepackage{exscale}
\usepackage[scaled=0.86]{helvet}

\usepackage[scr=euler,scrscaled=1.1,bb=txof,bbscaled=1.1,cal=boondox,calscaled=1.1]{mathalfa}
\renewcommand{\mathbf}{\bm}
\renewcommand{\mathit}[1]{\mathscr #1}
\renewcommand{\mathfrak}[1]{{\textsc{\upshape #1}}}
\renewcommand{\mathtt}[1]{\scalebox{1.15}{\bf \texttt{\upshape#1}}}
\renewcommand{\emph}[1]{\mathtt{#1}}
\renewcommand{\mathrm}[1]{\scalebox{1.01}{\textbf{\upshape #1}}}
\usepackage{float}

\usepackage[justification=centering]{caption}

\numberwithin{equation}{section}
\numberwithin{theorem}{section}

\usepackage[normalem]{ulem}
\usepackage[square,comma]{natbib}

\usepackage{mparhack}

\usepackage{keyval}
\usepackage{totcount}
\newtotcounter{citnum} 
\def\oldbibitem{} \let\oldbibitem=\bibitem
\def\bibitem{\stepcounter{citnum}\oldbibitem}

\usepackage[verbose]{backref}
\backrefsetup{verbose=false}
\renewcommand*{\backref}[1]{}
\renewcommand*{\backrefalt}[4]{[{\tiny%
    \ifcase #1 \textsl{Not cited}%
          \or \textsl{Cited on page}~\textcolor{BrickRed}{#2}%
          \else \textsl{Cited on pages}~\textcolor{BrickRed}{#2}%
    \fi%
    }]}

\usepackage{epigraph}

\usepackage{fancyhdr}
\author{\small\scshape S\lowercase{teven} D\lowercase{uplij}}

\address{
Yantai Research Institute,
Harbin Engineering University,
265615 Yantai, China
\\and\\
Center for Information Technology,
University of M\"unster,
48149 M\"unster,
Germany
}
\email{\small \sf douplii@uni-muenster.de;
sduplij@gmail.com;
http://www.uni-muenster.de/IT.StepanDouplii}

\title{\large\bfseries\scshape
H\lowercase{igher power polyadic group rings}}

\date{\textit{of start} August 15, 2025. \textit{Date}:
\textit{of completion}
October 15, 2025.
\mbox{}\hskip 1.16em
\textit{Total}:
18 
references
}

\renewcommand{\refname}{\textsc{References}}

\let\origsection\section
\renewcommand{\section}[1]{\sectionmark{#1}\origsection{#1}}
\let\origsubsection\subsection
\renewcommand{\subsection}[1]{\subsectionmark{#1}\origsubsection{#1}}

\makeatletter
\renewenvironment{thebibliography}[1]{%
  \@xp\origsection\@xp*\@xp{\refname}%
  \normalfont\footnotesize\labelsep .9em\relax
  \renewcommand\theenumiv{\arabic{enumiv}}\let\p@enumiv\@empty
  \vspace*{-5pt}
  \list{\@biblabel{\theenumiv}}{\settowidth\labelwidth{\@biblabel{#1}}%
    \leftmargin\labelwidth \advance\leftmargin\labelsep
    \usecounter{enumiv}}%
  \sloppy \clubpenalty\@M \widowpenalty\clubpenalty
  \sfcode`\.=\@m
}{%
  \def\@noitemerr{\@latex@warning{Empty `thebibliography' environment}}%
  \endlist
}
\makeatother

\subjclass[2010]{11A07, 11A67, 16S34, 17A40, 17A42, 20N15, 20C05}
\keywords{augmentation map, group ring, arity, polyadic structure, ternary group, polyadic zero, querelement}

\begin{document}

\mbox{}
\vskip 1.8cm
\begin{abstract}

\noindent This paper introduces and systematically develops the theory of
polyadic group rings, a higher arity generalization of classical group rings
$\mathcal{R}[\mathsf{G}]$. We construct the fundamental operations of these
structures, defining the $\mathbf{m}_{r}$-ary addition and $\mathbf{n}_{r}%
$-ary multiplication for a polyadic group ring $\mathrm{R}^{[\mathbf{m}%
_{r},\mathbf{n}_{r}]}=\mathcal{R}^{[m_{r},n_{r}]}[\mathsf{G}^{[n_{g}]}]$ built
from a nonderived $(m_{r},n_{r})$-ring and a nonderived $n_{g}$-ary group. A central result is the
derivation of the \textquotedblleft quantization\textquotedblright\ conditions
that interrelate these arities, governed by the arity freedom principle, which
also extends to operations with higher polyadic powers. We establish key
algebraic properties, including conditions for total associativity and the
existence of a zero element and identity. The concepts of the polyadic
augmentation map and augmentation ideal are generalized, providing a bridge to
the classical theory. The framework is illustrated with explicit examples,
solidifying the theoretical constructions. This work establishes a new
foundation in ring theory with potential applications in cryptography and
coding theory, as evidenced by recent schemes utilizing polyadic structures.

\end{abstract}

\maketitle



\thispagestyle{empty}
\mbox{}
\vspace{-0.5cm}
\tableofcontents
\newpage

\pagestyle{fancy}

\addtolength{\footskip}{15pt}

\renewcommand{\sectionmark}[1]{%
\markboth{
{ \scshape #1}}{}}

\renewcommand{\subsectionmark}[1]{%
\markright{
\mbox{\;}\\[5pt]
\textmd{#1}}{}}

\fancyhead{}
\fancyhead[EL,OR]{\leftmark}
\fancyhead[ER,OL]{\rightmark}
\fancyfoot[C]{\scshape -- \textcolor{BrickRed}{\thepage} --}


\renewcommand\headrulewidth{0.5pt}
\fancypagestyle {plain1}{ %
\fancyhf{}
\renewcommand {\headrulewidth }{0pt}
\renewcommand {\footrulewidth }{0pt}
}

\fancypagestyle{plain}{ %
\fancyhf{}
\fancyhead[C]{\scshape S\lowercase{teven} D\lowercase{uplij} \hskip 0.7cm \MakeUppercase{Polyadic Hopf algebras and quantum groups}}
\fancyfoot[C]{\scshape - \thepage  -}
\renewcommand {\headrulewidth }{0pt}
\renewcommand {\footrulewidth }{0pt}
}

\fancypagestyle{fancyref}{ %
\fancyhf{} 
\fancyhead[C]{\scshape R\lowercase{eferences} }
\fancyfoot[C]{\scshape -- \textcolor{BrickRed}{\thepage} --}
\renewcommand {\headrulewidth }{0.5pt}
\renewcommand {\footrulewidth }{0pt}
}

\fancypagestyle{emptyf}{
\fancyhead{}
\fancyfoot[C]{\scshape -- \textcolor{BrickRed}{\thepage} --}
\renewcommand{\headrulewidth}{0pt}
}
\thispagestyle{emptyf}

\section{\textsc{Introduction}}

The theory of group rings, which constructs a ring $\mathcal{R}[\mathsf{G}]$
from a given ring $\mathcal{R}$ and a group $\mathsf{G}$, is a cornerstone of
modern algebra. Its applications permeate various fields, including
representation theory, homological algebra, and algebraic topology
\cite{bovdi,passman,sehgal}. The standard construction leverages the binary
operations of the constituent ring and group to define an associative algebra,
providing a rich framework for studying the interplay between ring-theoretic
and group-theoretic properties.

A significant and modern generalization of classical algebraic arises from
increasing the arity of their fundamental operations. This leads to the theory
of polyadic algebraic structures \cite{dor3,pos}, where operations map $n$
elements to a single one, with $n\geq3$. This framework reveals phenomena
absent in the binary case; for instance, polyadic groups ($n$-ary groups) can
exist without a unique identity element or inverses in the classical sense,
with their structure governed by the more general concept of a querelement
\cite{dor3}. Similarly, polyadic rings, defined by an $m$-ary addition and an
$n$-ary multiplication linked by generalized distributivity laws, exhibit a
more complex and nuanced structure \cite{lee/but}.

While the theories of binary group rings and polyadic structures are
individually well-established, their synthesis the theory of polyadic group
rings remains largely unexplored. Constructing such an object, denoted
$\mathcal{R}^{[m_{r},n_{r}]}[\mathsf{G}^{[n_{g}]}]$ from an $(m_{r},n_{r}%
)$-ring and an $n_{g}$-ary group, presents fundamental challenges. The arities
of the initial structures are not independent; they are constrained by the
requirement that the resulting object must itself be a ring-like structure
with well-defined $\mathbf{m}_{r}$-ary addition and $\mathbf{n}_{r}$-ary
multiplication. This interplay is governed by what has been termed the arity
freedom principle \cite{duplij2022}, leading to "quantization" conditions that
determine the admissible arities $\mathbf{m}_{r}$ and $\mathbf{n}_{r}$ of the
resulting polyadic group ring.

In this article, we introduce and develop the theory of polyadic group rings.
Our primary objective is to generalize the classical construction to the
higher arity setting, establishing its foundational properties. The main
contributions of this work are as follows:

$\bullet$ We provide a rigorous definition of a polyadic group ring, formally
constructing its $\mathbf{m}_{r}$-ary addition and $\mathbf{n}_{r}$-ary
multiplication operations, carefully accounting for the arities of the
underlying ring and group.

$\bullet$ We derive the precise "quantization" conditions that link the
arities $(\mathbf{m}_{r},\mathbf{n}_{r})$ of the group ring to the arities
$(m_{r},n_{r})$ of the initial ring and $n_{g}$ of the initial group,
including the novel case of operations with higher polyadic powers.

$\bullet$ We establish key properties of these structures, proving under which
conditions the polyadic group ring is totally associative and possesses
analogues of a zero element and, when applicable, an identity.

$\bullet$ We define and investigate the concepts of the polyadic augmentation
map and the polyadic augmentation ideal, generalizing central tools from the
classical theory.

$\bullet$ We illustrate the theory with concrete, non-trivial examples
involving nonderived polyadic rings and finite polyadic groups, explicitly
computing products and demonstrating the workings of the constructed operations.

This work not only broadens the landscape of ring theory by introducing a new
class of algebraic objects but also provides a framework for future
investigations into their representation theory, homology, and other
invariants. Furthermore, the complex, non-binary operations inherent to
polyadic group rings present a promising algebraic platform for applications
in coding theory \cite{berlekamp,ric/urb} and post-quantum cryptography
\cite{men/oor/van}. The intricate structure of these systems, particularly the
convoluted multiplication defined by higher-arity group laws, could underpin
the development of new families of non-linear codes and form the basis for
multivariate-based encryption schemes or key exchange protocols resistant to
quantum cryptanalysis. This opens a new chapter in the study of higher arity
algebraic structures and their potential for practical computation and security.

A compelling demonstration of this practical potential can be found in
\cite{dup/guo}. The authors construct a novel encryption and decryption
procedure that directly leverages polyadic algebraic structures alongside
signal processing methods \cite{oppenheim}, which represents a tangible and
promising application of polyadic theory to cryptography, moving beyond purely
theoretical constructs. Its emergence strongly validates the timeliness and
relevance of foundational research into polyadic group rings, suggesting that
the structures formalized in this work may serve as the bedrock for future
cryptographic innovations and other applied systems.

\section{\textsc{Preliminaries}}

Here we present the notation and the general properties of polyadic structures
(for more details and references, consult \cite{duplij2022}).

Let $S^{\times n}$ be $n$-fold Cartesian product of a non-empty set $S$.
Elements of the form $(x_{1},\ldots,x_{n})\in S^{\times n}$ are termed polyads
or $n$-tuples $(\mathfrak{x})$. An $n$-tuple consisting of identical elements
is denoted $(x^{n})$. A polyadic operation (or $n$-ary operation) is defined
as a mapping ${\mu}_{n}:S^{\times n}\rightarrow S$, denoted by ${\mu}%
_{n}[\mathfrak{x}]$. A polyadic structure $\langle S\mid{\mu}_{n_{i}}\rangle$
consists of a set $S$ that is closed under a family of polyadic operations
${\mu}_{n_{i}}$.

The fundamental one-operation polyadic structure is the $n$-ary magma
$\mathcal{M}=\langle S\mid{\mu}_{n}\rangle$. The imposition of additional
axioms results in various group-like structures. For instance, a polyadically
associative magma constitutes an $n$-ary semigroup $\mathcal{S}_{n}=\langle
S\mid{\mu}_{n}\mid assoc\rangle$. Polyadic associativity is defined through
the invariance relation ${\mu}_{n}[\mathfrak{x},{\mu}_{n}[\mathfrak{y}%
],\mathfrak{z}]=invariant$, wherein the embedded multiplication may appear in
any of the $n-1$ positions (resulting in $n-1$ relations), which enables the
omission of parentheses in compositions. The polyads $\mathfrak{x}%
,\mathfrak{y},\mathfrak{z}$ have appropriate lengths such that the total
number of elements is $2n-1$. This iterated product
\begin{equation}
{\mu}_{n}^{\circ\ell_{\mu}}[\mathfrak{x}]=\overset{\ell_{\mu}}{\overbrace
{{\mu}_{n}[{\mu}_{n}[\ldots{\mu}_{n}[\mathfrak{x}]]]}},\quad\mathfrak{x}\in
S^{\ell_{\mu}(n-1)+1}, \label{ml}%
\end{equation}
where $\ell_{\mu}$ denotes the number of multiplication operations. From
(\ref{ml}), a fundamental distinction between polyadic and conventional binary
($n=2$) structures arises: the length $w_{\mu}(n)$ of a word in a composition
of $n$-ary multiplications is not arbitrary but quantized, assuming only the
admissible values indicating that multiplication is possible only for%
\begin{equation}
L^{\text{admiss}}(n,\ell_{\mu})=\ell_{\mu}(n-1)+1, \label{wl}%
\end{equation}
elements. This viewpoint facilitates the classification of polyadic operations
into two categories: those iterated from binary or lower-arity operations and
those that are noniterated, or equivalently, derived and nonderived.
Obviously, the latter are of more interesting to investigate.

We now recall the definitions of key elements in polyadic structures. For an
element $x\in S$, its $\ell_{\mu}$-polyadic power (or higher polyadic power)
is defined by
\begin{equation}
x^{\langle\ell_{\mu}\rangle}={\mu}_{n}^{\circ\ell_{\mu}}[x^{\ell_{\mu}%
(n-1)+1}], \label{xl}%
\end{equation}
which, in the binary case $n=2$, yields $x^{\langle\ell_{\mu}\rangle}%
=x^{\ell_{\mu}+1}$, differing by unity from the conventional power.

A polyadic idempotent $x_{id}$ (if existent) satisfies
\begin{equation}
x_{id}^{\langle\ell_{\mu}\rangle}=x_{id},\ \ \ \ x_{id}\in S. \label{id}%
\end{equation}

A polyadic zero $z$ is uniquely defined by the $n-1$ conditions
\begin{equation}
{\mu}_{n}[z,\mathfrak{x}]=z,\ \ \ \ \mathfrak{x}\in S^{n-1}, \label{z}%
\end{equation}
with $z$ positioned in any of the $n$ argument slots. A polyadic nilpotent
element $x_{nil}$ is defined by
\begin{equation}
x_{nil}^{\langle\ell_{\mu}\rangle}=z,\quad x_{nil}\in S. \label{mx}%
\end{equation}

A neutral $(n-1)$-polyad $\mathfrak{e}$ satisfies
\begin{equation}
{\mu}_{n}[x,\mathfrak{e}]=x,\quad\mathfrak{e}\in S^{n-1},
\end{equation}
which is typically non-unique. If all components of the neutral polyad are
identical, $\mathfrak{e}=e^{n-1}$, then
\begin{equation}
{\mu}_{n}[x,e^{n-1}]=x, \label{e1}%
\end{equation}
and $e$ is termed an identity of $\langle S\mid{\mu}_{n}\rangle$; it may
appear in any of the $n$ positions within the operation. From (\ref{z}) with
$\mathfrak{x}=z^{n-1}$ and (\ref{e1}) with $x=e$, it follows that both the
polyadic zero $z$ and the identity $e$ are idempotents satisfying (\ref{id}).
Certain exotic polyadic structures may lack idempotents, a zero, or an
identity altogether, or may feature multiple identities \cite{duplij2022}.

In the polyadic case ($n\geq3$), the notion of invertibility is not linked to
the identity (\ref{e1}) but is determined by the querelement $\bar{x}=\bar
{x}(x)$, defined via the $n-1$ relations \cite{dor3}
\begin{equation}
{\mu}_{n}[\bar{x},x^{n-1}]=x,\quad x\in S, \label{mq}%
\end{equation}
which must hold for $\bar{x}$ in each of the $n$ possible positions. Such an
element $x$ is termed polyadically invertible. If every element of an $n$-ary
semigroup $\mathcal{S}_{n}$ is polyadically invertible, then $\mathcal{S}_{n}$
constitutes an $n$-ary (polyadic) group $\mathcal{G}_{n}=\langle S\mid{\mu
}_{n}\mid assoc\rangle$. Notably, the presence of an identity is not a
prerequisite for polyadic groups.

Structures endowed with two polyadic operations fall within the class of
ring-like polyadic structures \cite{lee/but}. A polyadic ring, or
$(m,n)$-ring, $\mathcal{R}_{m,n}=\langle S\mid{\nu}_{m},{\mu}_{n}\rangle$,
consists of a non-empty set $S$ equipped with an $m$-ary addition ${\nu}%
_{m}:S^{m}\rightarrow S$ and an $n$-ary multiplication ${\mu}_{n}%
:S^{n}\rightarrow S$, such that: $\langle S\mid{\nu}_{m}\mid assoc\mid
comm\rangle$ forms an $m$-ary commutative group and $\langle S\mid{\mu}%
_{n}\mid assoc\rangle$ forms an $n$-ary semigroup. The operations ${\nu}_{m}$
and ${\mu}_{n}$ are interconnected by the following $n$-ary distributivity
relations: \cite{duplij2022}%
\begin{align}
&  {\mu}_{n}[{\nu}_{m}[x_{1},\ldots,x_{m}],y_{2},y_{3},\ldots,y_{n}%
]\nonumber\\
&  ={\nu}_{m}[{\mu}_{n}[x_{1},y_{2},y_{3},\ldots,y_{n}],{\mu}_{n}[x_{2}%
,y_{2},y_{3},\ldots,y_{n}],\ldots,\mathbf{\mu}_{n}[x_{m},y_{2},y_{3}%
,\ldots,y_{n}]], \label{dis1}%
\end{align}%
\begin{align}
&  {\mu}_{n}[y_{1},{\nu}_{m}[x_{1},\ldots,x_{m}],y_{3},\ldots,y_{n}%
]\nonumber\\
&  ={\nu}_{m}[{\mu}_{n}[y_{1},x_{1},y_{3},\ldots,y_{n}],{\mu}_{n}[y_{1}%
,x_{2},y_{3},\ldots,y_{n}],\ldots,\mathbf{\mu}_{n}[y_{1},x_{m},y_{3}%
,\ldots,y_{n}]], \label{dis2}%
\end{align}%
\begin{align}
&  \vdots\nonumber\\
&  {\mu}_{n}[y_{1},y_{2},\ldots,y_{n-1},{\nu}_{m}[x_{1},\ldots,x_{m}%
]]\nonumber\\
&  ={\nu}_{m}[{\mu}_{n}[y_{1},y_{2},\ldots,y_{n-1},x_{1}],{\mu}_{n}%
[y_{1},y_{2},\ldots,y_{n-1},x_{2}],\ldots,{\mu}_{n}[y_{1},y_{2},\ldots
,y_{n-1},x_{m}]], \label{dis3}%
\end{align}
where $x_{i},y_{j}\in S$, $i=1,\ldots,m$, $j=1,\ldots,n$.

If not all distributivity relations (\ref{dis1})--(\ref{dis3}) or
associativity relations hold, the ring is designated as partial (in contrast
to total), giving rise to a multitude of possible polyadic ring variants.
Further details are in \cite{duplij2022}.

\section{\textsc{Binary group rings}}

Here recall in brief the main constructions of the binary group rings in the
standard approach \cite{bovdi,passman,sehgal,zal/mik,mil/seh} and then present
them in the \textquotedblleft polyadic\textquotedblright\ language, which will
make their generalization to the novel higher arity approach more clear and transparent.

The most natural way to construct from two given one-set algebraic structures
$\mathit{A}\left(  1\right)  $ and $\mathit{B}\left(  1\right)  $ a new
two-set algebraic structure $\mathit{C}\left(  2\right)  $ is considering
formal combinations of elements from $\mathit{B}\left(  1\right)  $ having
\textquotedblleft weights\textquotedblright\ from $\mathit{A}\left(  1\right)
$ being \textquotedblleft more linear\textquotedblright, which is usually
denoted $\mathit{A}\left(  1\right)  \mathit{B}\left(  1\right)  $. Without
considering a product in $\mathit{A}\left(  1\right)  \mathit{B}\left(
1\right)  $ the result is only a free module-like structure in which
$\mathit{A}\left(  1\right)  $ plays a role of \textquotedblleft
scalars\textquotedblright, while $\mathit{B}\left(  1\right)  $ being the
\textquotedblleft basis\textquotedblright, taking the staring analogy with a
vector space. Further various definitions of a multiplication in
$\mathit{A}\left(  1\right)  \mathit{B}\left(  1\right)  \mathit{\ }$lead to
different algebra-like structures with nontrivial properties, which are
usually denoted as $\mathit{C}\left(  2\right)  =\mathit{A}\left(  1\right)
\left[  \mathit{B}\left(  1\right)  \right]  $. An equivalent approach to the
latter is consideration of the set of mappings $\mathit{B}\left(  1\right)
\rightarrow\mathit{A}\left(  1\right)  $ formally multiplied by
\textquotedblleft scalars\textquotedblright\ from $\mathit{A}\left(  1\right)
$ with the pointwise addition and the product as a convolution, commonly
denoted as $\mathit{A}\left(  1\right)  ^{\mathit{B}\left(  1\right)  }$. We
will exploit the first definition for higher arity generalizations. Typically,
the role of \textquotedblleft scalars\textquotedblright\ from $\mathit{A}%
\left(  1\right)  $ is played by rings, fields, etc., and for the role of
\textquotedblleft vectors\textquotedblright\ from $\mathit{B}\left(  1\right)
$ one takes semigroups, monoids, groups, loops, and so on.

In the simplest case, $\mathit{A}\left(  1\right)  $ is an associative ring
$\mathcal{R}$ (having the underlying set $R$, with the possible zero $0_{R}$
and unit $1_{R}$), and $\mathit{B}\left(  1\right)  $ is a group $\mathsf{G}$
(having the underlying set $G$, with the identity $\mathsf{e}_{\mathsf{G}}$),
and $\mathcal{R}\mathsf{G}$ can be built as follows (in the standard notation
\cite{zal/mik,passman,mil/seh}).

\begin{definition}
A free $\mathcal{R}$-module $\mathcal{R}\mathsf{G}$ with the basis $\left\{
\mathsf{g}\mid\mathsf{g}\in\mathsf{G}\right\}  $ is the set of finite formal
sums%
\begin{equation}
\sum_{\mathsf{g}\in\mathsf{G}}r_{\mathsf{g}}\bullet\mathsf{g}%
,\ \ \ r_{\mathsf{g}}\in R,\ \ \mathsf{g}\in G, \label{rg}%
\end{equation}
which endowed with the left-\textquotedblleft componentwise\textquotedblright%
\ addition%
\begin{equation}
\left(  \sum_{\mathsf{g}\in\mathsf{G}}r_{\mathsf{g}}\bullet\mathsf{g}\right)
+\left(  \sum_{\mathsf{g}\in\mathsf{G}}r_{\mathsf{g}}^{\prime}\bullet
\mathsf{g}\right)  =\left(  \sum_{\mathsf{g}\in\mathsf{G}}\left(
r_{\mathsf{g}}+r_{\mathsf{g}}^{\prime}\right)  \bullet\mathsf{g}\right)
\label{rrg}%
\end{equation}
and left-\textquotedblleft scalar\textquotedblright\ multiplication%
\begin{equation}
\lambda\left(  \sum_{\mathsf{g}\in\mathsf{G}}r_{\mathsf{g}}\bullet
\mathsf{g}\right)  =\left(  \sum_{\mathsf{g}\in\mathsf{G}}\left(
\lambda\,r_{\mathsf{g}}\right)  \bullet\mathsf{g}\right)  ,\ \ \lambda\in R.
\label{lr}%
\end{equation}

\end{definition}

Obviously, $\mathsf{G}\subset\mathcal{R}\mathsf{G}$, because for every
$\mathsf{h}\in G$ one can choose $r_{\mathsf{h}}=1_{R}$ and $r_{\mathsf{g}%
}=0_{R}$, if $\mathsf{g}\neq\mathsf{h}$. Thus, each element of $\mathcal{R}%
\mathsf{G}\mathcal{\ }$can be treated as a finite sum of such $r_{\mathsf{g}%
}\bullet\mathsf{g}$ for which $r_{\mathsf{g}}\neq0$, and the subset of such
$\mathsf{g\in}G_{\text{\textsl{supp}}}\subset G$ is the support of $\sum
r_{\mathsf{g}}\bullet\mathsf{g}$, i.e. $G_{\text{\textsl{supp}}}%
=$\textsl{support}$\left(  \sum r_{\mathsf{g}}\bullet\mathsf{g}\right)  $. In
general, if both underlying sets $R$ and $G$ are finite with $\left\vert
R\right\vert =N_{R}$ and $\left\vert G\right\vert =N_{\mathsf{G}}$, the total
number of elements in $\left\vert \mathcal{R}\mathsf{G}\right\vert
=N_{R\mathsf{G}}$ is $N_{R\mathsf{G}}=\left(  N_{R}\right)  ^{N_{\mathsf{G}}}$.

The product of sums (\ref{rg}) cannot be defined in left-\textquotedblleft
componentwise\textquotedblright\ way as the addition (\ref{rrg}), because it
ignores the group structure at all. Instead, consider the
both-\textquotedblleft componentwise\textquotedblright\ product of two terms
of different sums (\ref{rg}) and define their product by the natural way
$\left(  r_{\mathsf{g}}\bullet\mathsf{g}\right)  \left(  r_{\mathsf{h}%
}^{\prime}\bullet\mathsf{h}\right)  =\left(  r_{\mathsf{g}}r_{\mathsf{h}%
}^{\prime}\right)  \bullet\left(  \mathsf{gh}\right)  $. Then the
multiplication of elements from $\mathcal{R}\mathsf{G}$ can be presented as
follows (after reordering the terms)%
\begin{equation}
\left(  \sum_{\mathsf{g}\in G}r_{\mathsf{g}}\bullet\mathsf{g}\right)  \left(
\sum_{\mathsf{h}\in G}r_{\mathsf{h}}^{\prime}\bullet\mathsf{h}\right)
=\left(  \sum_{\mathsf{g}\in G}\sum_{\mathsf{h}\in G}\left(  r_{\mathsf{g}%
}r_{\mathsf{h}}^{\prime}\right)  \bullet\left(  \mathsf{gh}\right)  \right)
=\left(  \sum_{\mathsf{u}\in G}r_{\mathsf{u}}^{\prime\prime}\bullet
\mathsf{u}\right)  , \label{rgh}%
\end{equation}
where%
\begin{equation}
r_{\mathsf{u}}^{\prime\prime}=\sum_{\mathsf{g}\in G}\sum_{\mathsf{h}\in
G}\left(  r_{\mathsf{g}}r_{\mathsf{h}}^{\prime}\right)  |_{\mathsf{gh}%
=\mathsf{u}},\ \ \ \mathsf{g},\mathsf{h},\mathsf{u}\in G, \label{ru}%
\end{equation}
or in terms of one sum%
\begin{equation}
r_{\mathsf{u}}^{\prime\prime}=\sum_{\mathsf{g}\in G}\left(  r_{\mathsf{g}%
}r_{\mathsf{g}^{-1}\mathsf{u}}^{\prime}\right)  . \label{gu}%
\end{equation}

The product (\ref{rgh}) is associative and satisfies distributivity with
respect to the addition (\ref{rrg}).

\begin{definition}
The free module $\mathcal{R}\mathsf{G}$ (\ref{rg})--(\ref{lr}) endowed with
the product (\ref{rgh})--(\ref{ru}) becomes a ring which is called a group
ring denoted by $\mathrm{R}=\mathcal{R}\left[  \mathsf{G}\right]  $.
\end{definition}

If the initial ring is a field $\mathcal{R}=\mathcal{K}$, then $\mathcal{K}%
\left[  \mathsf{G}\right]  $ is called a group algebra over $\mathcal{K}$,
being a vector space over $\mathcal{K}$ with the dimension $\left\vert
G\right\vert $, if the group $\mathsf{G}$ is finite. Note that the group
algebra concept is the starting point for representation theory (see, e.g.
\cite{cur/rei,kirillov,erd/hol}).

Numerous properties of group rings were considered in
\cite{bovdi,passman,sehgal,zal/mik,mil/seh}, and refs therein.

Let us rewrite the main definitions of the binary group rings (\ref{rg}%
--(\ref{ru}) in the more detail \textquotedblleft polyadic\textquotedblright%
\ functional notation by writing the operations manifestly.

The initial (binary) group is the algebraic structure having one-set $\left\{
\mathsf{g}\right\}  =G$ and one main associative binary operation
$\mu_{\mathsf{G}}=\mu_{\mathsf{G}}^{\left[  \mathsf{2}\right]  }:G\times
G\rightarrow G$, as $\mathsf{G}=\mathsf{G}^{\left[  2\right]  }=\left\langle
G\mid\mu_{\mathsf{G}}^{\left[  \mathsf{2}\right]  },\left(  \ \right)
^{-1}\right\rangle $ together with the identity $\mathsf{e}_{\mathsf{G}}$
satisfying $\mu_{\mathsf{G}}\left[  \mathsf{e}_{\mathsf{G}},\mathsf{g}\right]
=\mu_{\mathsf{G}}\left[  \mathsf{g},\mathsf{e}_{\mathsf{G}}\right]
=\mathsf{g}$ and the inverse $\left(  \ \right)  ^{-1}$ such that
$\mu_{\mathsf{G}}\left[  \mathsf{g},\mathsf{g}^{-1}\right]  =\mu_{\mathsf{G}%
}\left[  \mathsf{g}^{-1},\mathsf{g}\right]  =\mathsf{e}_{\mathsf{G}}$ for all
$\mathsf{g}\in G$. The initial ring $\mathcal{R}$ is the one-set $\left\{
r\right\}  =R$ algebraic structure $\mathcal{R}=\mathcal{R}^{\left[
2,2\right]  }=\left\langle R\mid\nu_{R}^{\left[  2\right]  },\mu_{R}^{\left[
2\right]  }\right\rangle $ endowed by two binary operations: addition $\nu
_{R}=\nu_{R}^{\left[  2\right]  }:R\times R\rightarrow R$ and multiplication
$\mu_{R}=\mu_{R}^{\left[  2\right]  }:R\times R\rightarrow R$ which satisfy
distributivity, such that $\left\langle R\mid\nu_{R}^{\left[  2\right]
}\right\rangle $ is the additive semigroup, and $\left\langle R\mid\mu
_{R}^{\left[  2\right]  }\right\rangle $ is the multiplicative group.

Note that by (\ref{lr}) a new operation (scalar multiplication) in
$\mathcal{R}\mathsf{G}$ is quietly defined, which is possible, because any
ring is a module over itself. Therefore, at first glance, the resulting group
ring $\mathcal{R}\left[  \mathsf{G}\right]  $ is a 2-set and 4-operation
algebraic structure, but we will see that it is more complicated.

Using this notation we present the definition (\ref{rg})--(\ref{ru}) of the
group ring $\mathrm{R}=\mathcal{R}\left[  \mathsf{G}\right]  $ in the
\textquotedblleft polyadic\textquotedblright\ functional form. Instead of the
abstract sum $\sum$ in (\ref{rg}) we use the concrete summation
$\mathbf{\Sigma}$ by indices manifestly (see, e.g. \cite{zal/mik}). In this
way, an $\alpha$th element $\mathbf{r}\left(  \alpha\right)  $ of the group
ring $\mathrm{R}=\mathcal{R}\left[  \mathsf{G}\right]  $ (with the underlying
set $\mathfrak{R}$) can be written as the formal sum%
\begin{equation}
\mathrm{r}\left(  \alpha\right)  =\mathrm{r}\left(  \overrightarrow
{r}_{\overrightarrow{\mathsf{g}}}\left(  \alpha\right)  ,\overrightarrow
{\mathsf{g}}\right)  =\underset{i}{\mathbf{\Sigma}}\ r_{\mathsf{g}_{i}}\left(
\alpha\right)  \bullet\mathsf{g}_{i},\ \ \ \ r_{\mathsf{g}_{i}}\left(
\alpha\right)  \in R,\ \ \mathsf{g}_{i}\in G,\ \ \mathrm{r}\left(
\alpha\right)  \in\mathfrak{R}, \label{ra}%
\end{equation}
where $\overrightarrow{r}_{\overrightarrow{\mathsf{g}}}\left(  \alpha\right)
=\left(  r_{\mathsf{g}_{1}}\left(  \alpha\right)  ,r_{\mathsf{g}_{2}}\left(
\alpha\right)  ,\ldots\right)  $, $\overrightarrow{\mathsf{g}}=\left(
\mathsf{g}_{1},\mathsf{g}_{2},\ldots\right)  $ are the \textquotedblleft
vectors\textquotedblright\ of ring elements and group elements, respectively.
In case the group ring $\mathcal{R}\left[  \mathsf{G}\right]  $ is finite,
$\alpha=1,2\ldots\left\vert \mathfrak{R}\right\vert $.

Denote the binary addition in $\mathrm{R}=\mathcal{R}\left[  \mathsf{G}%
\right]  $ (\ref{rrg}) as $\mathbf{\nu}_{\mathfrak{R}}^{\left[  \mathbf{2}%
\right]  }:\mathfrak{R}\times\mathfrak{R}\rightarrow\mathfrak{R}$, then%
\begin{equation}
\mathbf{\nu}_{\mathfrak{R}}^{\left[  \mathbf{2}\right]  }\left[
\mathrm{r}\left(  \overrightarrow{r}_{\overrightarrow{\mathsf{g}}}\left(
\alpha_{1}\right)  ,\overrightarrow{\mathsf{g}}\right)  ,\mathrm{r}\left(
\overrightarrow{r}_{\mathsf{\vec{g}}}\left(  \alpha_{2}\right)
,\overrightarrow{\mathsf{g}}\right)  \right]  =\underset{i}{\mathbf{\Sigma}%
}\ \nu_{R}^{\left[  2\right]  }\left[  r_{\mathsf{g}_{i}}\left(  \alpha
_{1}\right)  ,r_{\mathsf{g}_{i}}\left(  \alpha_{2}\right)  \right]
\bullet\mathsf{g}_{i},\ \ \ \ r_{\mathsf{g}_{i}}\left(  \alpha_{1,2}\right)
\in R,\ \ \mathsf{g}_{i}\in G. \label{n2}%
\end{equation}

The binary multiplication in $\mathrm{R}=\mathcal{R}\left[  \mathsf{G}\right]
$ is denoted by $\mathbf{\mu}_{\mathfrak{R}}^{\left[  \mathbf{2}\right]
}:\mathfrak{R}\times\mathfrak{R}\rightarrow\mathfrak{R}$, using (\ref{rgh}) we
define the convolution-like operation%
\begin{align}
\mathbf{\mu}_{\mathfrak{R}}^{\left[  \mathbf{2}\right]  }\left[
\mathrm{r}\left(  \overrightarrow{r}_{\overrightarrow{\mathsf{g}}}\left(
\alpha_{1}\right)  ,\overrightarrow{\mathsf{g}}\right)  ,\mathrm{r}\left(
\overrightarrow{r}_{\mathsf{\vec{g}}}\left(  \alpha_{2}\right)
,\overrightarrow{\mathsf{g}}\right)  \right]   &  =\underset{i}{\mathbf{\Sigma
}}\underset{j}{\mathbf{\Sigma}}\ \mu_{R}^{\left[  2\right]  }\left[
r_{\mathsf{g}_{i}}\left(  \alpha_{1}\right)  ,r_{\mathsf{g}_{j}}\left(
\alpha_{2}\right)  \right]  \bullet\mu_{\mathsf{G}}^{\left[  \mathsf{2}%
\right]  }\left[  \mathsf{g}_{i},\mathsf{g}_{j}\right]  ,\label{m2}\\
r_{\mathsf{g}_{i,j}}\left(  \alpha_{1,2}\right)   &  \in R,\ \ \mathsf{g}%
_{i}\in G.\nonumber
\end{align}
Denote%
\begin{equation}
\mathsf{g}_{k}=\mu_{\mathsf{G}}^{\left[  \mathsf{2}\right]  }\left[
\mathsf{g}_{i},\mathsf{g}_{j}\right]  ,\ \ \ \ \ k=k\left(  i,j\right)
,\ \ \ \ \mathsf{g}_{i,j,k}\in G, \label{g}%
\end{equation}
then (\ref{rgh}) \textquotedblleft in components\textquotedblright\ takes the
form%
\begin{align}
\mathbf{\mu}_{\mathfrak{R}}^{\left[  \mathbf{2}\right]  }\left[
\mathrm{r}\left(  \overrightarrow{r}_{\overrightarrow{\mathsf{g}}}\left(
\alpha_{1}\right)  ,\overrightarrow{\mathsf{g}}\right)  ,\mathrm{r}\left(
\overrightarrow{r}_{\overrightarrow{\mathsf{g}}}\left(  \alpha_{2}\right)
,\overrightarrow{\mathsf{g}}\right)  \right]   &  =\underset{i}{\mathbf{\Sigma
}}\underset{j}{\mathbf{\Sigma}}\ \mu_{R}^{\left[  2\right]  }\left[
r_{\mathsf{g}_{i}}\left(  \alpha_{1}\right)  ,r_{\mathsf{g}_{j}}\left(
\alpha_{2}\right)  \right]  \bullet\mathsf{g}_{k\left(  i.j\right)  },\\
r_{\mathsf{g}_{i,j,k\left(  i,j\right)  }}\left(  \alpha_{1,2}\right)   &  \in
R,\ \ \mathsf{g}_{i}\in G.\nonumber
\end{align}

We can resolve (\ref{g}) with respect to $\mathsf{g}_{j}$, reorder indices and
present (\ref{m2}) as follows%
\begin{equation}
\mathbf{\mu}_{\mathfrak{R}}^{\left[  \mathbf{2}\right]  }\left[
\mathrm{r}\left(  \overrightarrow{r}_{\overrightarrow{\mathsf{g}}}\left(
\alpha_{1}\right)  ,\overrightarrow{\mathsf{g}}\right)  ,\mathrm{r}\left(
\overrightarrow{r}_{\overrightarrow{\mathsf{g}}}\left(  \alpha_{2}\right)
,\overrightarrow{\mathsf{g}}\right)  \right]  =\underset{k}{\mathbf{\Sigma}%
}\ r_{\mathsf{g}_{k}}^{\prime}\left(  \alpha_{1},\alpha_{2}\right)
\bullet\mathsf{g}_{k},
\end{equation}
where%
\begin{equation}
r_{\mathsf{g}_{k}}^{\prime}\left(  \alpha_{1},\alpha_{2}\right)  =\underset
{i}{\mathbf{\Sigma}}\ \mu_{R}^{\left[  2\right]  }\left[  r_{\mathsf{g}_{i}%
}\left(  \alpha_{1}\right)  ,r_{\mathsf{g}_{i}^{-1}\mathsf{g}_{k}}\left(
\alpha_{2}\right)  \right]  ,
\end{equation}
which is the component/function version of (\ref{gu}).

The set of sums (\ref{ra}) $\left\{  \mathrm{r}\right\}  =\mathfrak{R}$
together with associative addition (\ref{n2}) and multiplication (\ref{m2})
being distributive is the standard ring $\left\langle \mathfrak{R}%
\mid\mathbf{v}_{\mathfrak{R}}^{\left[  \mathbf{2}\right]  },\mathbf{\mu
}_{\mathfrak{R}}^{\left[  \mathbf{2}\right]  }\right\rangle $. However, the
group ring $\mathcal{R}\left[  \mathsf{G}\right]  $ is more than a ring,
because by definition it has the additional module-like operation, the
\textquotedblleft scalar\textquotedblright\ multiplication (\ref{lr}). Since
the \textquotedblleft scalars\textquotedblright\ are from the initial ring
$\mathcal{R}$, it is an internal ring product, because every ring is a module
by itself. But for the group ring $\mathcal{R}\left[  \mathsf{G}\right]  $
(\ref{lr}) is the unnoticed new operation (1-place action or $\mathcal{R}%
$-module) $\mathbf{\rho}_{\mathfrak{R}}^{\left[  1\right]  }:R\times
\mathfrak{R}\rightarrow\mathfrak{R}$, which is in the \textquotedblleft
polyadic\textquotedblright\ notation becomes (for $k$-place actions, see
\cite{duplij2022})%
\begin{equation}
\mathbf{\rho}_{\mathfrak{R}}^{\left[  \mathbf{1}\right]  }\left(  \lambda
\mid\mathrm{r}\left(  \overrightarrow{r}_{\overrightarrow{\mathsf{g}}}\left(
\alpha\right)  ,\overrightarrow{\mathsf{g}}\right)  \right)  =\underset
{i}{\mathbf{\Sigma}}\ \mu_{R}^{\left[  2\right]  }\left[  \lambda
,r_{\mathsf{g}_{i}}\left(  \alpha\right)  \right]  \bullet\mathsf{g}%
_{i},\ \ \ \ \lambda,r_{\mathsf{g}_{i}}\left(  \alpha\right)  \in
R,\ \ \mathsf{g}_{i}\in G. \label{r1}%
\end{equation}

In this way, we could think that the formal \textquotedblleft
polyadic\textquotedblright\ definition of the group ring, as 2 set and 3
operation algebraic structure $\mathrm{R}=\left\langle \mathfrak{R}%
,R,G\mid\mathbf{\nu}_{\mathfrak{R}}^{\left[  \mathbf{2}\right]  },\mathbf{\mu
}_{\mathfrak{R}}^{\left[  \mathbf{2}\right]  },\mathbf{\rho}_{\mathfrak{R}%
}^{\left[  \mathbf{1}\right]  }\right\rangle $, $\lambda\in R$. Nevertheless,
after adding the operations (addition $\nu_{R}^{\left[  2\right]  }$ and
multiplication $\mu_{R}^{\left[  2\right]  }$) in $\mathcal{R}$, we have (by
classification of \cite{duplij2022})

\begin{definition}
The binary group ring $\mathcal{R}\left[  \mathsf{G}\right]  $ is the 3 set
and 6 operation algebra-like structure%
\begin{equation}
\mathrm{R}=\left\langle \mathfrak{R},R,G\mid\mathbf{\nu}_{\mathfrak{R}%
}^{\left[  \mathbf{2}\right]  },\mathbf{\mu}_{\mathfrak{R}}^{\left[
\mathbf{2}\right]  },\mathbf{\rho}_{\mathfrak{R}}^{\left[  \mathbf{1}\right]
}\mid\nu_{R}^{\left[  2\right]  },\mu_{R}^{\left[  2\right]  }\mid
\mu_{\mathsf{G}}^{\left[  \mathsf{2}\right]  }\right\rangle , \label{rrrg}%
\end{equation}
where the first 3 operations are defined in (\ref{n2}), (\ref{m2}) (\ref{r1}), respectively.
\end{definition}

\section{\textsc{Polyadic group rings}}

Now we generalize the group ring concept to higher arity case and introduce a
novel algebraic structure, the polyadic group ring. By doing so, we take
advantage of the \textquotedblleft arity freedom principle\textquotedblright%
\ \cite{duplij2022}: in any algebraic structure, initial arities of all its
operations can be taken arbitrary, then the structural constraints appear from
the general dependences leading to \textquotedblleft quantization
rules\textquotedblright\ which forbidden certain combination of arities. In
this way, polyadic structures can have exotic properties, for instance $n$-ary
groups without identity or with many identities, polyadic fields without zero
or/and without unit, and so on \cite{duplij2022}, which leads to revision even
standard theorems and statements.

In the abstract setting from beginning of the previous section, the initial
polyadic algebraic structures now carry arbitrary arities $\mathit{A}^{\left[
arity_{\mathit{A}}\right]  }\left(  1\right)  $ and $\mathit{B}^{\left[
arity_{\mathit{B}}\right]  }\left(  1\right)  $. Then the resulting polyadic
algebra-like structure $\mathit{C}^{\left[  arity_{\mathit{C}}\right]
}(2)=\mathit{A}^{\left[  arity_{\mathit{A}}\right]  }\left(  1\right)  \left[
\mathit{B}^{\left[  arity_{\mathit{B}}\right]  }\left(  1\right)  \right]  $
will possess the specific arity $arity_{\mathit{C}}$ which is determined by
the operations framework.

So instead of the binary ring $\mathcal{R}$ as $\mathit{A}\left(  1\right)  $
and group $\mathsf{G}$ as $\mathit{B}\left(  1\right)  $, we consider $\left(
n_{r},m_{r}\right)  $-ring $\mathcal{R}^{\left[  m_{r},n_{r}\right]  }$ as
$\mathit{A}^{\left[  arity_{\mathit{A}}\right]  }\left(  1\right)  $ and
$n_{g}$-ary group $\mathsf{G}^{\left[  n_{g}\right]  }$ as $\mathit{B}%
^{\left[  arity_{\mathit{B}}\right]  }\left(  1\right)  $, where%
\begin{equation}
\mathcal{R}^{\left[  m_{r},n_{r}\right]  }=\left\langle R\mid\nu_{R}^{\left[
m_{r}\right]  },\mu_{R}^{\left[  n_{r}\right]  }\right\rangle , \label{mn}%
\end{equation}
with totally associative $m_{r}$-ary addition $\nu_{R}^{\left[  m_{r}\right]
}:R^{\times m_{r}}\rightarrow R$ and $n_{r}$-ary multiplication $\mu
_{R}^{\left[  n_{r}\right]  }:R^{\times n_{r}}\rightarrow R$, which satisfy
polyadic distributivity.

If polyadic zero $z_{R}$ and polyadic identity (or unity) $e_{R}$ in
$\mathcal{R}^{\left[  m_{r},n_{r}\right]  }$ exist, they satisfy additive
neutrality and multiplicative absorption, and multiplicative neutrality,
respectively%
\begin{align}
\nu_{R}^{\left[  m_{r}\right]  }\left[  r,\overset{m_{r}-1}{\overbrace
{z_{R},\ldots,z_{R}}}\right]   &  =r,\label{zr}\\
\mu_{R}^{\left[  m_{r}\right]  }\left[  r_{1},r_{2},\ldots,r_{m_{r}-1}%
,z_{R}\right]   &  =z_{R}\label{zrm}\\
\mu_{R}^{\left[  n_{r}\right]  }\left[  r,\overset{n_{r}-1}{\overbrace
{e_{R},\ldots,e_{R}}}\right]   &  =r,\ \ \ \ r,r_{i},z_{R},e_{R}\in
R,\label{er}%
\end{align}
where $z_{R}$ and $e_{R}$ can be on any places. The polyadic zero $z_{R}$ in
$\mathcal{R}^{\left[  m_{r},n_{r}\right]  }$ is not necessary for
$\left\langle R\mid\nu_{R}^{\left[  m_{r}\right]  }\right\rangle $ to be a
$m_{r}$-ary additive group for $m_{r}\geq3$ (which is impossible for binary
groups), but the additive querelement is important \cite{dup2017a}.
Nevertheless, one can adjoin the extraneous polyadic zero $\dot{z}_{R}\notin
R$ externally by extending the underlying set $R$ of the initial ring
$\mathcal{R}^{\left[  m_{r},n_{r}\right]  }$ as follows%
\begin{equation}
\dot{R}=R\cup\left\{  \dot{z}_{R}\right\}  ,\label{r0}%
\end{equation}
where $\dot{z}_{R}$ satisfies the needed standard relations (\ref{zr}) and
(\ref{zrm}).

The neutral element $e_{R}$ being a polyadic identity, has nothing with the
multiplicative invertibility in the ring $\mathcal{R}^{\left[  m_{r}%
,n_{r}\right]  }$. Nevertheless, some elements of $\mathcal{R}^{\left[
m_{r},n_{r}\right]  }$ can be invertible, which means that for them there
exists a polyadic analog of multiplicative inverse, the querelement $\bar{r}$
defined by \cite{dor3}%
\begin{equation}
\mu_{R}^{\left[  n_{r}\right]  }\left[  \bar{r},\overset{n_{r}-1}%
{\overbrace{r,\ldots,r}}\right]  =r,\ \ \ \ \bar{r},r\in R, \label{qm}%
\end{equation}
where $\bar{r}$ can be on any place and $n_{r}\geq3$. By analogy with the
binary ring, a polyadic ring without unity can be called a polyadic rng, a
non-unital polyadic ring or pseudo-ring. The simplest $\left(  n_{r}%
,m_{r}\right)  $-rng example is $2\mathbb{Z}$.

We denote in $\mathcal{R}^{\left[  m_{r},n_{r}\right]  }$ the subset of
multiplicatively invertible elements (sometimes called units) by $\emph{U}%
_{R}\subset R$ which, in the binary case, is called a unit group
$\mathit{U}^{\left[  \mathfrak{n}_{u}\right]  }\left(  \mathcal{R}^{\left[
m_{r},n_{r}\right]  }\right)  $. In the polyadic case the set $\mathit{U}%
^{\left[  \mathfrak{n}_{u}\right]  }$ should be the $\mathfrak{n}_{u}$-ary
group with $\mathfrak{n}_{u}=n_{r}$.

The $n_{g}$-ary group is%

\begin{equation}
\mathsf{G}^{\left[  \mathsf{n}_{g}\right]  }=\left\langle G\mid\mu
_{\mathsf{G}}^{\left[  \mathsf{n}_{g}\right]  },\overline{\left(  \ \right)
}\right\rangle , \label{gn}%
\end{equation}
with $\mathsf{n}_{g}$-ary multiplication $\mathsf{n}_{g}:G^{\times
\mathsf{n}_{g}}\rightarrow G$, and each element $\mathsf{g}$ has the analog of
inverse, its querelement $\mathsf{\bar{g}}$ obeying%
\begin{equation}
\mu_{\mathsf{G}}\left[  \mathsf{\bar{g}},\overset{\mathsf{n}_{g}-1}%
{\overbrace{\mathsf{g},\ldots,\mathsf{g}}}\right]  =\mathsf{g}%
,\ \ \ \ \mathsf{\bar{g}},\mathsf{g}\in G, \label{gg}%
\end{equation}
where $\mathsf{\bar{g}}$ can be on any place.

If in $\mathsf{G}^{\left[  \mathsf{n}_{g}\right]  }$ the polyadic identity
$\mathsf{e}_{\mathsf{G}}$ exists (which is not necessary for $\mathsf{n}%
_{g}\geq3$), it satisfies multiplicative neutrality
\begin{equation}
\mu_{\mathsf{G}}\left[  \mathsf{g},\overset{\mathsf{n}_{g}-1}{\overbrace
{\mathsf{e}_{\mathsf{G}},\ldots,\mathsf{e}_{\mathsf{G}}}}\right]
=\mathsf{g},\ \ \ \ \mathsf{e}_{\mathsf{G}},\mathsf{g}\in G,\label{eg}%
\end{equation}
where $\mathsf{g}$ can be on any place. For more details and definitions, see
\cite{duplij2022}.

Let us construct from $\left(  m_{r},n_{r}\right)  $-ring $\mathcal{R}%
^{\left[  m_{r},n_{r}\right]  }$ (\ref{mn}) and $n_{g}$-ary group
$\mathsf{G}^{\left[  n_{g}\right]  }$ (\ref{gn}) the polyadic group ring
$\mathrm{R}^{\left[  \mathbf{m}_{r},\mathbf{n}_{r}\right]  }$ with the same
underlying set of the formal sums (\ref{ra}) $\mathfrak{R}=\left\{
\mathrm{r}\right\}  $%
\begin{equation}
\mathrm{R}^{\left[  \mathbf{m}_{r},\mathbf{n}_{r}\right]  }=\mathcal{R}%
^{\left[  n_{r},m_{r}\right]  }\left[  \mathsf{G}^{\left[  \mathsf{n}%
_{g}\right]  }\right]  , \label{rm}%
\end{equation}
but now obeying $\mathbf{m}_{r}$-ary addition $\mathbf{\nu}^{\left[
\mathbf{m}_{r}\right]  }:\mathfrak{R}^{\times\mathbf{m}_{r}}\rightarrow
\mathfrak{R}$ and $\mathbf{n}_{r}$-ary multiplication $\mathbf{\mu}^{\left[
\mathbf{n}_{r}\right]  }:\mathfrak{R}^{\times\mathbf{n}_{r}}\rightarrow
\mathfrak{R}$. Using the \textquotedblleft arity freedom
principle\textquotedblright\ \cite{duplij2022} and analogy with the binary
case (\ref{rrrg}) we have

\begin{definition}
The polyadic group ring is the 3 set and 6 operation polyadic algebra-like
structure%
\begin{equation}
\mathrm{R}^{\left[  \mathbf{m}_{r},\mathbf{n}_{r}\right]  }=\left\langle
\mathfrak{R},R,G\mid\mathbf{\nu}_{\mathfrak{R}}^{\left[  \mathbf{m}%
_{r}\right]  },\mathbf{\mu}_{\mathfrak{R}}^{\left[  \mathbf{n}_{r}\right]
},\mathbf{\rho}_{\mathfrak{R}}^{\left[  \mathbf{k}_{\rho}\right]  }\mid\nu
_{R}^{\left[  m_{r}\right]  },\mu_{R}^{\left[  n_{r}\right]  }\mid
\mu_{\mathsf{G}}^{\left[  \mathsf{n}_{g}\right]  }\right\rangle . \label{rmn}%
\end{equation}

\end{definition}

Now we generalize the binary operations $\mathbf{\nu}_{\mathfrak{R}}^{\left[
\mathbf{2}\right]  }$ (\ref{n2}), $\mathbf{\mu}_{\mathfrak{R}}^{\left[
\mathbf{2}\right]  }$ (\ref{m2}) and $\mathbf{\rho}_{\mathfrak{R}}^{\left[
\mathbf{1}\right]  }$ (\ref{r1}) to higher arity setting, implying that the
arities of initial $\left(  m_{r},n_{r}\right)  $-ring $\mathcal{R}^{\left[
m_{r},n_{r}\right]  }$ (\ref{mn}) and $n_{g}$-ary group $\mathsf{G}^{\left[
n_{g}\right]  }$ (\ref{gn}) are given.

\begin{definition}
The $\mathbf{m}_{r}$-ary addition $\mathbf{\nu}^{\left[  \mathbf{m}%
_{r}\right]  }$ can be defined by analogy with (\ref{n2})
left-\textquotedblleft componentwise\textquotedblright\ by%
\begin{align}
\mathbf{\nu}_{\mathfrak{R}}^{\left[  \mathbf{m}_{r}\right]  }\left[
\mathrm{r}\left(  \overrightarrow{r}_{\overrightarrow{\mathsf{g}}}\left(
\alpha_{1}\right)  ,\overrightarrow{\mathsf{g}}\right)  ,\ldots,\mathrm{r}%
\left(  \overrightarrow{r}_{\mathsf{\vec{g}}}\left(  \alpha_{\mathbf{m}_{r}%
}\right)  ,\overrightarrow{\mathsf{g}}\right)  \right]   &  =\underset
{i}{\mathbf{\Sigma}}\ \nu_{R}^{\left[  m_{r}\right]  }\left[  r_{\mathsf{g}%
_{i}}\left(  \alpha_{1}\right)  ,\ldots,r_{\mathsf{g}_{i}}\left(
\alpha_{m_{r}}\right)  \right]  \bullet\mathsf{g}_{i},\label{mm}\\
r_{\mathsf{g}_{i}}\left(  \alpha_{1,\ldots,\mathbf{m}_{r}}\right)   &  \in
R,\ \ \mathsf{g}_{i}\in G,\ \ \ \mathrm{r}\left(  \overrightarrow
{r}_{\overrightarrow{\mathsf{g}}}\left(  \alpha_{1,\ldots,\mathbf{m}_{r}%
}\right)  ,\overrightarrow{\mathsf{g}}\right)  \in\mathfrak{R}.\nonumber
\end{align}

\end{definition}

\begin{proposition}
The arity of addition in the polyadic group ring $\mathrm{R}^{\left[
\mathbf{m}_{r},\mathbf{n}_{r}\right]  }$ coincides with the arity of addition
in the initial polyadic ring $\mathcal{R}^{\left[  n_{r},m_{r}\right]  }$ that
is%
\begin{equation}
\mathbf{m}_{r}=m_{r}, \label{mrm}%
\end{equation}
if in both sides of (\ref{mm}) there is one polyadic operation (addition in
$\mathrm{R}^{\left[  \mathbf{m}_{r},\mathbf{n}_{r}\right]  }$ and addition in
$\mathcal{R}^{\left[  m_{r},n_{r}\right]  }$).
\end{proposition}

\begin{proof}
The statement (\ref{mrm}) directly follows from the construction (\ref{mm}).
\end{proof}

\begin{remark}
\label{rem-oper}In the polyadic framework and from \textquotedblleft arity
freedom principle\textquotedblright\ \cite{duplij2022}, it follows that number
of operations in both sides of (\ref{mm}) can be different, such that the
arities of addition in the initial ring and the group ring may also differ,
but the total number of ring elements in brackets should remain the same.
\end{remark}

\begin{remark}
\label{rem-q}Denote the number of $m_{r}$-ary additions in in the initial
$\left(  m_{r},n_{r}\right)  $-ring $\mathcal{R}^{\left[  m_{r},n_{r}\right]
}$ by $\ell_{m}$, being actually the polyadic power, and their composition by
$\left(  \nu_{R}^{\left[  m_{r}\right]  }\right)  ^{\circ\ell_{m}}$, where the
total number of arguments is not arbitrary, as in the binary case, but
\textquotedblleft quantized\textquotedblright\ becoming%
\begin{equation}
\ell_{m}\left(  m_{r}-1\right)  +1. \label{ll}%
\end{equation}

\end{remark}

\begin{definition}
The higher (polyadic) power $\mathbf{m}_{r}$-ary addition in the polyadic
group ring $\mathrm{R}^{\left[  \mathbf{m}_{r},\mathbf{n}_{r}\right]  }$ can
be defined by analogy with (\ref{mm}) left-\textquotedblleft
componentwise\textquotedblright\ by%
\begin{align}
&  \mathbf{\nu}_{\mathfrak{R}}^{\left[  \mathbf{m}_{r}\right]  }\left[
\mathrm{r}\left(  \overrightarrow{r}_{\overrightarrow{\mathsf{g}}}\left(
\alpha_{1}\right)  ,\overrightarrow{\mathsf{g}}\right)  ,\ldots,\mathrm{r}%
\left(  \overrightarrow{r}_{\overrightarrow{\mathsf{g}}}\left(  \alpha
_{\mathbf{m}_{r}}\right)  ,\overrightarrow{\mathsf{g}}\right)  \right]
\nonumber\\
&  =\ \underset{i}{\mathbf{\Sigma}}\ \left(  \nu_{R}^{\left[  m_{r}\right]
}\right)  ^{\circ\ell_{m}}\left[  r_{\mathsf{g}_{i}}\left(  \alpha_{1}\right)
,\ldots,r_{\mathsf{g}_{i}}\left(  \alpha_{\ell_{m}\left(  m_{r}-1\right)
+1}\right)  \right]  \bullet\mathsf{g}_{i},\label{mml}\\
&  r_{\mathsf{g}_{i}}\left(  \alpha_{1,\ldots,\ell_{m}\left(  m_{r}-1\right)
+1}\right)  \in R,\ \ \mathsf{g}_{i}\in G,\ \ \ \mathrm{r}\left(
\overrightarrow{r}_{\overrightarrow{\mathsf{g}}}\left(  \alpha_{1,\ldots
,\mathbf{m}_{r}}\right)  ,\overrightarrow{\mathsf{g}}\right)  \in
\mathfrak{R}.\nonumber
\end{align}

\end{definition}

Therefore, we have

\begin{theorem}
The arity of addition $\mathbf{m}_{r}$ in the polyadic group ring
$\mathrm{R}^{\left[  \mathbf{m}_{r},\mathbf{n}_{r}\right]  }$ which possesses
a higher polyadic power of $m_{r}$-ary addition is%
\begin{equation}
\mathbf{m}_{r}=\ell_{m}\left(  m_{r}-1\right)  +1. \label{mlm}%
\end{equation}

\end{theorem}

\begin{proof}
The statement (\ref{mlm}) follows from the construction (\ref{mml}),
\textit{Remarks} \ref{rem-oper}, \ref{rem-q} and the \textquotedblleft
quantization\textquotedblright\ condition (\ref{ll}).
\end{proof}

The multiplication in $\mathrm{R}^{\left[  \mathbf{m}_{r},\mathbf{n}%
_{r}\right]  }$ can be defined similarly to the binary case in functional
notation (\ref{m2})

\begin{definition}
In the polyadic group ring $\mathcal{R}^{\left[  m_{r},n_{r}\right]  }\left[
\mathsf{G}^{\left[  \mathsf{n}_{g}\right]  }\right]  $ (\ref{rm}) the
multiplication $\mathbf{\mu}_{\mathfrak{R}}^{\left[  \mathbf{n}_{r}\right]
}:\mathfrak{R}^{\times\mathbf{n}_{r}}\rightarrow\mathfrak{R}$ can be defined
by the both-\textquotedblleft componentwise\textquotedblright%
\ convolution-like operation%
\begin{align}
&  \mathbf{\mu}_{\mathfrak{R}}^{\left[  \mathbf{n}_{r}\right]  }\left[
\mathrm{r}\left(  \overrightarrow{r}_{\overrightarrow{\mathsf{g}}}\left(
\alpha_{1}\right)  ,\overrightarrow{\mathsf{g}}\right)  ,\ldots,\mathrm{r}%
\left(  \overrightarrow{r}_{\overrightarrow{\mathsf{g}}}\left(  \alpha
_{\mathbf{n}_{r}}\right)  ,\overrightarrow{\mathsf{g}}\right)  \right]
\nonumber\\
&  =\underset{i_{1}}{\mathbf{\Sigma}}\ \ldots\ \underset{i_{n_{r}}%
}{\mathbf{\Sigma\ }}\underset{j_{1}}{\mathbf{\Sigma}}\ \ldots\ \underset
{j_{\mathsf{n}_{g}}}{\mathbf{\Sigma}}\mu_{R}^{\left[  n_{r}\right]  }\left[
r_{\mathsf{g}_{i_{1}}}\left(  \alpha_{1}\right)  ,\ldots,r_{\mathsf{g}%
_{i_{n_{r}}}}\left(  \alpha_{n_{r}}\right)  \right]  \bullet\mu_{\mathsf{G}%
}^{\left[  \mathsf{n}_{g}\right]  }\left[  \mathsf{g}_{j_{1}},\ldots
,\mathsf{g}_{j_{\mathsf{n}_{g}}}\right]  ,\label{mng}\\
&  r_{\mathsf{g}_{i}}\left(  \alpha_{j}\right)  \in R,\ \ \mathsf{g}_{i}\in
G,\ \ \ \mathrm{r}\left(  \overrightarrow{r}_{\overrightarrow{\mathsf{g}}%
}\left(  \alpha_{j}\right)  ,\overrightarrow{\mathsf{g}}\right)
\in\mathfrak{R},\nonumber
\end{align}
if polyadic power of all operations is $1$.
\end{definition}

\begin{proposition}
The arities of multiplications defined by (\ref{mng}) in $\mathcal{R}^{\left[
n_{r},m_{r}\right]  }$, $\mathsf{G}^{\left[  \mathsf{n}_{g}\right]  }$ and
$\mathrm{R}^{\left[  \mathbf{m}_{r},\mathbf{n}_{r}\right]  }$ coincide%
\begin{equation}
\mathbf{n}_{r}=n_{r}=\mathsf{n}_{g}. \label{nnn}%
\end{equation}

\end{proposition}

\begin{proof}
The statement (\ref{nnn}) follows directly from both-\textquotedblleft
componentwise\textquotedblright\ convolution (\ref{mng}).
\end{proof}

The construction of the multiplication is more elaborate and interesting, if
we take into account \textit{Remarks} \ref{rem-oper}, \ref{rem-q} and use the
"arity freedom principle" \cite{duplij2022} (initial arities are taken
arbitrary). Indeed, let us denote polyadic powers of $\mu_{R}^{\left[
n_{r}\right]  }$ and $\mu_{\mathsf{G}}^{\left[  \mathsf{n}_{g}\right]  }$ by
$\ell_{n}$ and $\ell_{g}$, correspondingly, then we have

\begin{definition}
The polyadic group ring with higher polyadic power of multiplications is
defined by%
\begin{align}
&  \mathbf{\mu}_{\mathfrak{R}}^{\left[  \mathbf{n}_{r}\right]  }\left[
\mathrm{r}\left(  \overrightarrow{r}_{\overrightarrow{\mathsf{g}}}\left(
\alpha_{1}\right)  ,\overrightarrow{\mathsf{g}}\right)  ,\ldots,\mathrm{r}%
\left(  \overrightarrow{r}_{\overrightarrow{\mathsf{g}}}\left(  \alpha
_{\mathbf{n}_{r}}\right)  ,\overrightarrow{\mathsf{g}}\right)  \right]
\label{mmr}\\
&  =\underset{\ i_{1}}{\mathbf{\Sigma}}\ \ldots\ \underset{\ell_{n}\left(
n_{r}-1\right)  +1}{\mathbf{\Sigma\ }}\underset{\ \ \ j_{1}}%
{\ \ \mathbf{\Sigma}}\ \ldots\ \underset{\ell_{g}\left(  \mathsf{n}%
_{g}-1\right)  +1}{\mathbf{\Sigma}}\left(  \mu_{R}^{\left[  n_{r}\right]
}\right)  ^{\circ\ell_{n}}\left[  r_{\mathsf{g}_{i_{1}}}\left(  \alpha
_{1}\right)  ,\ldots,r_{\mathsf{g}_{i_{\ell_{n}\left(  n_{r}-1\right)  +1}}%
}\left(  \alpha_{\ell_{n}\left(  n_{r}-1\right)  +1}\right)  \right]
\nonumber\\
&  \bullet\left(  \mu_{\mathsf{G}}^{\left[  \mathsf{n}_{g}\right]  }\right)
^{\ell_{\mathsf{n}_{g}}}\left[  \mathsf{g}_{j_{1}},\ldots,\mathsf{g}%
_{j_{\ell_{g}\left(  \mathsf{n}_{g}-1\right)  +1}}\right]  ,\ \ r_{\mathsf{g}%
_{i}}\left(  \alpha_{j}\right)  \in R,\ \ \mathsf{g}_{i}\in G,\ \ \ \mathrm{r}%
\left(  \overrightarrow{r}_{\overrightarrow{\mathsf{g}}}\left(  \alpha
_{j}\right)  ,\overrightarrow{\mathsf{g}}\right)  \in\mathfrak{R}.\nonumber
\end{align}

\end{definition}

The \textquotedblleft quantization\textquotedblright\ conditions for
multiplications in $\mathcal{R}^{\left[  n_{r},m_{r}\right]  }$ and
$\mathsf{G}^{\left[  \mathsf{n}_{g}\right]  }$, analogous to those for
additions in (\ref{ll}), now arise from the equality of the total number of
arguments in the equation (\ref{mmr}).%
\begin{equation}
\ell_{n}\left(  n_{r}-1\right)  +1=\ell_{g}\left(  \mathsf{n}_{g}-1\right)
+1. \label{lll}%
\end{equation}

Thus, we have

\begin{theorem}
The arity of multiplication $\mathbf{n}_{r}$ in the polyadic group ring
$\mathrm{R}^{\left[  \mathbf{m}_{r},\mathbf{n}_{r}\right]  }$ (with higher
polyadic powers of multiplications in the initial polyadic ring $\mu
_{R}^{\left[  n_{r}\right]  }$ and the $\mathsf{n}_{g}$-ary group
$\mu_{\mathsf{G}}^{\left[  \mathsf{n}_{g}\right]  }$) is%
\begin{equation}
\mathbf{n}_{r}=\ell_{n}\left(  n_{r}-1\right)  +1=\ell_{g}\left(
\mathsf{n}_{g}-1\right)  +1. \label{nln}%
\end{equation}

\end{theorem}

\begin{proof}
The statement (\ref{nln}) follows from the construction (\ref{mmr}) and the
\textquotedblleft quantization\textquotedblright\ condition (\ref{lll}).
\end{proof}

Obviously, if all polyadic powers are equal to one $\ell_{n}=\ell_{g}=1$, then
all operations share the same arity (\ref{nnn}).

\begin{definition}
The polyadic group ring $\mathrm{R}^{\left[  \mathbf{m}_{r},\mathbf{n}%
_{r}\right]  }=\mathcal{R}^{\left[  m_{r},n_{r}\right]  }\left[
\mathsf{G}^{\left[  \mathsf{n}_{g}\right]  }\right]  $ (\ref{rm}) which has
initial multiplications $\mu_{R}^{\left[  n_{r}\right]  }$ and $\mu
_{\mathsf{G}}^{\left[  \mathsf{n}_{g}\right]  }$ of higher polyadic powers
(\ref{mmr}), is called a higher power polyadic group ring.
\end{definition}

\section{\textsc{Properties}}

Here we consider basic properties of the polyadic group rings, which are in
the higher arity case can be unusual and exotic.

In the binary case, the associativity of addition of the group ring
$\mathcal{R}\left[  \mathsf{G}\right]  $ trivially follows from the
associativity of addition in the initial ring $\mathcal{R}$ because of the the
left-\textquotedblleft componentwise\textquotedblright\ addition (\ref{rrg}).
The same conclusion is valid for polyadic additions $\mathbf{\nu
}_{\mathfrak{R}}^{\left[  \mathbf{m}_{r}\right]  }$ and $\nu_{R}^{\left[
m_{r}\right]  }$ with unit polyadic power (\ref{mm}).

\begin{proposition}
In case of higher polyadic powers (\ref{mml}) $\ell_{m}>1$ the total
associativity of addition in the polyadic group ring $\mathrm{R}^{\left[
\mathbf{m}_{r},\mathbf{n}_{r}\right]  }$ follows from the associativity of the
addition in the initial ring $\mathcal{R}^{\left[  m_{r},n_{r}\right]  }$, if
the strong inequality takes place%
\begin{equation}
\mathbf{m}_{r}>m_{r}.
\end{equation}

\end{proposition}

\begin{proof}
It follows from the \textquotedblleft quantization\textquotedblright%
\ condition (\ref{ll}) and the informal statement \textquotedblleft larger
brackets can be constructed from smaller brackets\textquotedblright.
\end{proof}

The connection between associativities of multiplications is more complicated
to prove.

\begin{theorem}
If the multiplication in the initial ring $\mathcal{R}^{\left[  m_{r}%
,n_{r}\right]  }$ is totally polyadic associative, then the polyadic group
ring $\mathrm{R}^{\left[  \mathbf{m}_{r},\mathbf{n}_{r}\right]  }$ is totally
associative multiplicatively, when all the arities of multiplication are equal%
\begin{equation}
\mathbf{n}_{r}=n_{r}=\mathsf{n}_{g}. \label{nn}%
\end{equation}

\end{theorem}

\begin{proof}
Using (\ref{mng}) in the notation (\ref{ra}), we compute $\mathbf{n}_{r}$
terms in the total associativity%
\begin{align}
\mathrm{r}\left(  \beta_{1}\right)   &  =\Sigma\ \mathbf{\mu}_{\mathfrak{R}%
}^{\left[  \mathbf{n}_{r}\right]  }\left[  \mathbf{\mu}_{\mathfrak{R}%
}^{\left[  \mathbf{n}_{r}\right]  }\left[  \mathrm{r}\left(  \alpha
_{1}\right)  ,\ldots,\mathrm{r}\left(  \alpha_{\mathbf{n}_{r}}\right)
\right]  ,\mathrm{r}\left(  \alpha_{\mathbf{n}_{r}+1}\right)  ,\ldots
,\mathrm{r}\left(  \alpha_{2\mathbf{n}_{r}-1}\right)  \right]  ,\\
\mathrm{r}\left(  \beta_{2}\right)   &  =\Sigma\ \mathbf{\mu}_{\mathfrak{R}%
}^{\left[  \mathbf{n}_{r}\right]  }\left[  \mathrm{r}\left(  \alpha
_{1}\right)  ,\mathbf{\mu}_{\mathfrak{R}}^{\left[  \mathbf{n}_{r}\right]
}\left[  \mathrm{r}\left(  \alpha_{2}\right)  ,\ldots,\mathrm{r}\left(
\alpha_{\mathbf{n}_{r}+1}\right)  \right]  ,\mathrm{r}\left(  \alpha
_{\mathbf{n}_{r}+2}\right)  ,\ldots,\mathrm{r}\left(  \alpha_{2\mathbf{n}%
_{r}-1}\right)  \right]  ,\\
&  \vdots\nonumber\\
\mathrm{r}\left(  \beta_{\mathbf{n}_{r}}\right)   &  =\Sigma\ \mathbf{\mu
}_{\mathfrak{R}}^{\left[  \mathbf{n}_{r}\right]  }\left[  \mathrm{r}\left(
\alpha_{1}\right)  ,\ldots,\mathrm{r}\left(  \alpha_{\mathbf{n}_{r}-1}\right)
,\mathbf{\mu}_{\mathfrak{R}}^{\left[  \mathbf{n}_{r}\right]  }\left[
\mathrm{r}\left(  \alpha_{\mathbf{n}_{r}}\right)  ,\ldots,\mathrm{r}\left(
\alpha_{2\mathbf{n}_{r}-1}\right)  \right]  \right]  ,
\end{align}
where $\Sigma$ denotes the sum by all corresponding internal indices. Then we
take into account the both-\textquotedblleft componentwise\textquotedblright%
\ convolution and total associativity of $\mathsf{n}_{g}$-ary group
$\mathsf{G}^{\left[  \mathsf{n}_{g}\right]  }$ to obtain%
\begin{align}
\mathrm{r}\left(  \beta_{1}\right)   &  =\Sigma\ \mu_{R}^{\left[
n_{r}\right]  }\left[  \mu_{R}^{\left[  n_{r}\right]  }\left[  r_{\mathsf{g}%
_{i_{1}}}\left(  \alpha_{1}\right)  ,\ldots,r_{\mathsf{g}_{i_{n_{r}}}}\left(
\alpha_{n_{r}}\right)  \right]  ,r_{\mathsf{g}_{i_{n_{r}}+1}}\left(
\alpha_{n_{r}+1}\right)  ,\ldots,r_{\mathsf{g}_{2i_{n_{r}}-1}}\left(
\alpha_{2n_{r}-1}\right)  \right] \nonumber\\
&  \bullet\left(  \mu_{\mathsf{G}}^{\left[  \mathsf{n}_{g}\right]  }\right)
^{\circ2}\left[  \mathsf{g}_{j_{1}},\ldots,\mathsf{g}_{2j_{\mathsf{n}_{g}}%
+1}\right]  ,\\
\mathrm{r}\left(  \beta_{2}\right)   &  =\Sigma\ \mu_{R}^{\left[
n_{r}\right]  }\left[  r_{\mathsf{g}_{i_{1}}}\left(  \alpha_{1}\right)
,\mu_{R}^{\left[  n_{r}\right]  }\left[  r_{\mathsf{g}_{i_{2}}}\left(
\alpha_{2}\right)  \ldots,r_{\mathsf{g}_{i_{n_{r}}+1}}\left(  \alpha_{n_{r}%
+1}\right)  \right]  ,r_{\mathsf{g}_{i_{n_{r}}+2}}\left(  \alpha_{n_{r}%
+2}\right)  ,\ldots,r_{\mathsf{g}_{2i_{n_{r}}-1}}\left(  \alpha_{2n_{r}%
-1}\right)  \right] \\
&  \bullet\left(  \mu_{\mathsf{G}}^{\left[  \mathsf{n}_{g}\right]  }\right)
^{\circ2}\left[  \mathsf{g}_{j_{1}},\ldots,\mathsf{g}_{2j_{\mathsf{n}_{g}}%
+1}\right]  ,\\
&  \vdots\\
\mathrm{r}\left(  \beta_{\mathbf{n}_{r}}\right)   &  =\Sigma\ \mu_{R}^{\left[
n_{r}\right]  }\left[  r_{\mathsf{g}_{i_{1}}}\left(  \alpha_{1}\right)
,\ldots,r_{\mathsf{g}_{i_{n_{r}-1}}}\left(  \alpha_{n_{r}-1}\right)  ,\mu
_{R}^{\left[  n_{r}\right]  }\left[  r_{\mathsf{g}_{i_{n_{r}}}}\left(
\alpha_{n_{r}}\right)  ,\ldots,r_{\mathsf{g}_{2i_{n_{r}}-1}}\left(
\alpha_{2n_{r}-1}\right)  \right]  \right] \\
&  \bullet\left(  \mu_{\mathsf{G}}^{\left[  \mathsf{n}_{g}\right]  }\right)
^{\circ2}\left[  \mathsf{g}_{j_{1}},\ldots,\mathsf{g}_{2j_{\mathsf{n}_{g}}%
+1}\right]  ,
\end{align}
where the (group dependence) terms on r.h.s. of the formal product $\left(
\bullet\right)  $ coincide after suitable rename of the summation indices. All
the terms on the l.h.s. are equal due to the total polyadic associativity in
the initial ring $\mathcal{R}^{\left[  m_{r},n_{r}\right]  }$. Thus, the
polyadic group ring $\mathrm{R}^{\left[  \mathbf{m}_{r},\mathbf{n}_{r}\right]
}=\mathcal{R}^{\left[  m_{r},n_{r}\right]  }\left[  \mathsf{G}^{\left[
\mathsf{n}_{g}\right]  }\right]  $ (\ref{rm}) is multiplicatively totally
associative for equal arities (\ref{nn}).
\end{proof}

\begin{theorem}
In case of higher polyadic powers (\ref{mml}) $\ell_{n}>1$ and/or $\ell_{g}>1$
the total associativity of multiplication in the polyadic group ring
$\mathrm{R}^{\left[  \mathbf{m}_{r},\mathbf{n}_{r}\right]  }=\mathcal{R}%
^{\left[  m_{r},n_{r}\right]  }\left[  \mathsf{G}^{\left[  \mathsf{n}%
_{g}\right]  }\right]  $ follows from the associativity of the addition in the
initial ring $\mathcal{R}^{\left[  m_{r},n_{r}\right]  }$, if the strong
inequalities take place%
\begin{equation}
\mathbf{n}_{r}>n_{r}\cup\mathbf{n}_{r}>n_{g}.
\end{equation}

\end{theorem}

\begin{proof}
It follows from the \textquotedblleft quantization\textquotedblright%
\ condition (\ref{lll}) and the informal consequence \textquotedblleft larger
brackets can be constructed from smaller brackets\textquotedblright.
\end{proof}

The polyadic distributivity in $\mathrm{R}^{\left[  \mathbf{m}_{r}%
,\mathbf{n}_{r}\right]  }$ is governed by polyadic distributivity of the
initial ring $\mathcal{R}^{\left[  m_{r},n_{r}\right]  }$, because the
addition is present only in the l.h.s. of the polyadic group ring elements
(\ref{rg}).

By the same reason, if the initial ring $\mathcal{R}^{\left[  m_{r}%
,n_{r}\right]  }$ has the polyadic zero $z_{R}$, then the group ring
$\mathrm{R}^{\left[  \mathbf{m}_{r},\mathbf{n}_{r}\right]  }$ has the polyadic
zero $\mathrm{z}_{\mathfrak{R}}$ of the form%
\begin{equation}
\mathrm{z}_{\mathfrak{R}}=z_{R}\bullet\underset{i}{\mathbf{\Sigma}%
}\ \mathsf{g}_{i},\ \ \ \ z_{R}\in R,\ \ \mathsf{g}_{i}\in G,\ \ \mathrm{z}%
_{\mathfrak{R}}\in\mathfrak{R}, \label{zrr}%
\end{equation}
such that no group elements appear with nonzero coefficients, since the finite
support, and such element is unique. Also, the zero $\mathrm{z}_{\mathfrak{R}%
}$ in $\mathrm{R}^{\left[  \mathbf{m}_{r},\mathbf{n}_{r}\right]  }$ is the
additive polyadic identity and is multiplicatively absorbing (for
$\mathcal{R}^{\left[  m_{r},n_{r}\right]  }$ see (\ref{zr}) and (\ref{zrm}))%
\begin{align}
\mathbf{\nu}_{\mathfrak{R}}^{\left[  \mathbf{m}_{r}\right]  }\left[
\mathrm{r}\left(  \alpha\right)  ,\overset{m_{r}-1}{\overbrace{\mathrm{z}%
_{\mathfrak{R}},\ldots,\mathrm{z}_{\mathfrak{R}}}}\right]   &  =\mathrm{r}%
\left(  \alpha\right)  ,\\
\mu_{R}^{\left[  \mathbf{n}_{r}\right]  }\left[  \mathrm{r}\left(  \alpha
_{1}\right)  ,\mathrm{r}\left(  \alpha_{2}\right)  ,\ldots,\mathrm{r}\left(
\alpha_{n_{r}-1}\right)  ,\mathrm{z}_{\mathfrak{R}}\right]   &  =\mathrm{z}%
_{\mathfrak{R}},\ \ \ \ \ \mathrm{r}\left(  \alpha\right)  ,\mathrm{r}\left(
\alpha_{i}\right)  ,\mathrm{z}_{\mathfrak{R}}\in\mathfrak{R}.
\end{align}

In the polyadic case, the identity of multiplication is only a neutral element
(\ref{eg}) and has no connection with invertibility (see, e.g.
\cite{duplij2022}). If the initial $\left(  m_{r},n_{r}\right)  $-ring
$\mathcal{R}^{\left[  m_{r},n_{r}\right]  }$ has the polyadic identity $e_{R}$
(\ref{er}) and the identity of the $n_{g}$-ary group $\mathsf{G}^{\left[
\mathsf{n}_{g}\right]  }$ (\ref{gn}) is $\mathsf{e}_{\mathsf{G}}$, then the
trivial polyadic identity $\mathrm{e}_{\mathfrak{R}}$ of the group ring
$\mathrm{R}^{\left[  \mathbf{m}_{r},\mathbf{n}_{r}\right]  }$ is%
\begin{equation}
\mathrm{e}_{\mathfrak{R}}=e_{R}\bullet\mathsf{e}_{\mathsf{G}}%
,\ \ \ \ \ \ \ e_{R}\in R,\ \ \mathsf{e}_{\mathsf{G}}\in G,\ \ \mathrm{e}%
_{\mathfrak{R}}\in\mathfrak{R}, \label{ee}%
\end{equation}
such that in the sum (\ref{ra}) one coefficient from $\mathcal{R}^{\left[
m_{r},n_{r}\right]  }$ at the group identity $\mathsf{e}_{\mathsf{G}}$ is
$e_{R}$, while others are equal to $z_{R}$ and therefore are not written here.

The invertibility properties of polyadic structures are governed not by
neutral elements, but by querelements \cite{duplij2022}. So for elements from
multiplicative $\mathfrak{n}_{u}$-ary unit group $\mathit{U}^{\left[
\mathfrak{n}_{u}\right]  }\left[  \mathcal{R}^{\left[  m_{r},n_{r}\right]
}\right]  $ (the subset $\emph{U}_{R}\subset R$ having the querelement
$\bar{r}$), and the $n_{g}$-ary group $\mathsf{G}^{\left[  \mathsf{n}%
_{g}\right]  }$ (\ref{gn}) having the querelement $\mathsf{\bar{g}}$
(\ref{gg}) for each $\mathsf{g}\in\mathsf{G}$, we can formulate

\begin{definition}
In the group ring $\mathrm{R}^{\left[  \mathbf{m}_{r},\mathbf{n}_{r}\right]
}$ the querelement $\overline{\mathrm{r}}\left(  \alpha\right)  $ for some
elements $\mathrm{r}\left(  \alpha\right)  \in\mathfrak{R}$ is defined by%
\begin{equation}
\overline{\mathrm{r}}\left(  \alpha\right)  =\underset{i}{\mathbf{\Sigma}%
}\ \bar{r}_{\mathsf{g}_{i}}\left(  \alpha\right)  \bullet\mathsf{\bar{g}}%
_{i},\ \ \ \ \ \ \ \bar{r}\in\emph{U},\ \ \mathsf{\bar{g}}_{i},\mathsf{g}%
_{i}\in G,\ \ \overline{\mathrm{r}}\in\mathfrak{R},
\end{equation}
where in the simplest case $\bar{r}$ and $\mathsf{\bar{g}}$ are defined in
(\ref{qm}) and (\ref{gg}), respectively, if all arities coincide (\ref{nn}).
\end{definition}

In the nontrivial approach, the r.h.s. can contain also more general elements
for which one should solve system of equations in each concrete case. We
denote the subset of multiplicatively invertible elements in the group ring
$\mathrm{R}^{\left[  \mathbf{m}_{r},\mathbf{n}_{r}\right]  }$ (sometimes
called group units) by $\emph{U}_{\mathfrak{R}}\in\mathfrak{R}$ which should
form the $\mathbf{n}_{u}$-ary group $\mathit{U}^{\left[  \mathbf{n}%
_{u}\right]  }\left(  \mathrm{R}^{\left[  \mathbf{m}_{r},\mathbf{n}%
_{r}\right]  }\right)  $, or a polyadic unit group.

\begin{definition}
The polyadic augmentation map $\mathbf{\varepsilon}:\mathcal{R}^{\left[
m_{r},n_{r}\right]  }\left[  \mathsf{G}^{\left[  \mathsf{n}_{g}\right]
}\right]  \rightarrow\mathcal{R}^{\left[  m_{r},n_{r}\right]  }$ can be
defined for $\alpha$th element (\ref{ra}) of $\mathrm{R}^{\left[
\mathbf{m}_{r},\mathbf{n}_{r}\right]  }$ as follows%
\begin{equation}
\underset{i=1}{\mathbf{\Sigma}^{i_{\max}}}\ r_{\mathsf{g}_{i}}\left(
\alpha\right)  \bullet\mathsf{g}_{i}\mapsto\left(  \nu_{R}^{\left[
m_{r}\right]  }\right)  ^{\circ\ell_{i}}\left[  r_{\mathsf{g}_{1}}\left(
\alpha\right)  ,\ldots,r_{\mathsf{g}_{\ell_{i}\left(  m_{r}-1\right)  +1}%
}\left(  \alpha\right)  \right]  ,\ \ \ \ r_{\mathsf{g}_{i}}\left(
\alpha\right)  \in\emph{U}_{R},\ \ \mathsf{g}_{i}\in G. \label{aug}%
\end{equation}

\end{definition}

\begin{remark}
Note that on the l.h.s. of (\ref{aug}) we have the formal sum $\Sigma$ by $i$,
if finite, then till $i_{\max}$, while on the r.h.s. the sum becomes $m_{r}%
$-ary addition in the initial $\left(  m_{r},n_{r}\right)  $-ring
$\mathcal{R}^{\left[  m_{r},n_{r}\right]  }$, therefore we have the
\textquotedblleft quantization\textquotedblright\ condition for the polyadic
augmentation%
\begin{equation}
i_{\max}=i_{\max}\left(  \ell_{i},m_{r}\right)  =\ell_{i}\left(
m_{r}-1\right)  +1, \label{il}%
\end{equation}
where $\ell_{i}$ is the polyadic power of the initial ring $\mathcal{R}%
^{\left[  m_{r},n_{r}\right]  }$ addition.
\end{remark}

\begin{definition}
The kernel of the polyadic augmentation map is called a polyadic augmentation
ideal and is defined by setting coefficients sum in (\ref{aug}) equal to the
polyadic zero (\ref{zrr}), as follows%
\begin{equation}
\ker\varepsilon=\left\langle \mathfrak{R}\mid\left(  \nu_{R}^{\left[
m_{r}\right]  }\right)  ^{\circ\ell_{i}}\left[  r_{\mathsf{g}_{1}}\left(
\alpha\right)  ,\ldots,r_{\mathsf{g}_{\ell_{i}\left(  m_{r}-1\right)  +1}%
}\left(  \alpha\right)  \right]  =\mathrm{z}_{\mathfrak{R}}\right\rangle
,\ \ \ \ r_{\mathsf{g}_{i}}\left(  \alpha\right)  \in\emph{U}_{R}%
,\ \ \mathsf{g}_{i}\in G. \label{ke}%
\end{equation}

\end{definition}

The polyadic augmentation map $\mathbf{\varepsilon}$ preserves addition and
multiplication without changing the arities and maps the corresponding
identities in $\mathcal{R}^{\left[  m_{r},n_{r}\right]  }\left[
\mathsf{G}^{\left[  \mathsf{n}_{g}\right]  }\right]  $ and $\mathcal{R}%
^{\left[  m_{r},n_{r}\right]  }$.

\section{\textsc{Examples}}

Let us present simple, but nontrivial examples of the polyadic group rings
$\mathrm{R}^{\left[  \mathbf{m}_{r},\mathbf{n}_{r}\right]  }=\mathcal{R}%
^{\left[  m_{r},n_{r}\right]  }\left[  \mathsf{G}^{\left[  \mathsf{n}%
_{g}\right]  }\right]  $ and list their main properties. First, we present in
detail the example of the polyadic power equal to one, then briefly show
higher polyadic powers in addition and multiplication separately.

\begin{example}
\label{exam-R23}We take for the initial ring the commutative nonderived
$\left(  2,3\right)  $-ring with the underlying set $R=\mathfrak{j}\mathbb{Z}$
($\mathfrak{j}^{2}=-1$), operations are in $\mathbb{C}$. Now $\nu_{R}^{\left[
2\right]  }$ and $\mu_{R}^{\left[  3\right]  }$are usual addition and product.
Note that $\left\langle \mathfrak{j}\mathbb{Z}\mid\nu_{R}^{\left[  2\right]
}\right\rangle $ is binary group with respect to addition, and $\left\langle
\mathfrak{j}\mathbb{Z}\mid\mu_{R}^{\left[  3\right]  }\right\rangle $ is not a
ternary group, but only a ternary semigroup, because there is no
multiplicative querelement for each $r\in R$. Obviously, that $\mathcal{R}%
^{\left[  2,3\right]  }$ is unitless. Note that polyadic distributivity
follows from the binary distributivity in $\mathbb{Z}$. Therefore,%
\begin{equation}
\mathcal{R}^{\left[  2,3\right]  }=\left\langle \mathfrak{j}\mathbb{Z}\mid
\nu_{R}^{\left[  2\right]  },\mu_{R}^{\left[  3\right]  },z_{R}\right\rangle
\label{r23}%
\end{equation}
is a commutative nonderived $\left(  2,3\right)  $-ring without multiplicative
neutral element and $z_{R}=\mathfrak{j}0=0$.

Let $\mathit{C}_{3}=\left\langle \left\{  e,a,a^{2}\right\}  \mid a^{3}%
=a^{0}=e\right\rangle $ be the cyclic group of order $3$ with the identity $e$
and one generator $a$. We take for the initial polyadic group $\mathsf{G}%
^{\left[  \mathsf{n}_{g}\right]  }$ the finite set of $2\times2$ antidiagonal
symbolic matrices $G=\mathfrak{adiag}\left(  \mathit{C}_{3},\mathit{C}%
_{3}\right)  $, which is closed with respect to triple matrix multiplication.
The element of $G$ can be presented in the form%
\begin{equation}
\mathsf{g}_{i}=g\left(  m,n\right)  =\left(
\begin{array}
[c]{cc}%
0 & a^{m}\\
a^{n} & 0
\end{array}
\right)  , \label{gi}%
\end{equation}
where $m,n\in\mathbb{Z}\operatorname{mod}3$, such that the cardinality
$\left\vert G\right\vert =9$, and the manifest form of the elements (\ref{gi})
are%
\begin{equation}%
\begin{tabular}
[c]{ccc}%
$\mathsf{g}_{1}=\left(
\begin{array}
[c]{cc}%
0 & e\\
e & 0
\end{array}
\right)  ,$ & $\mathsf{g}_{2}=\left(
\begin{array}
[c]{cc}%
0 & a\\
e & 0
\end{array}
\right)  ,$ & $\mathsf{g}_{3}=\left(
\begin{array}
[c]{cc}%
0 & a^{2}\\
e & 0
\end{array}
\right)  ,$\\
$\mathsf{g}_{4}=\left(
\begin{array}
[c]{cc}%
0 & e\\
a & 0
\end{array}
\left(
\begin{array}
[c]{cc}%
0 & a^{2}\\
e & 0
\end{array}
\right)  \right)  ,$ & $\mathsf{g}_{5}=\left(
\begin{array}
[c]{cc}%
0 & a\\
a & 0
\end{array}
\right)  ,$ & $\mathsf{g}_{6}=\left(
\begin{array}
[c]{cc}%
0 & a^{2}\\
a & 0
\end{array}
\right)  ,$\\
$\mathsf{g}_{7}=\left(
\begin{array}
[c]{cc}%
0 & e\\
a^{2} & 0
\end{array}
\right)  ,$ & $\mathsf{g}_{8}=\left(
\begin{array}
[c]{cc}%
0 & a\\
a^{2} & 0
\end{array}
\right)  ,$ & $\mathsf{g}_{9}=\left(
\begin{array}
[c]{cc}%
0 & a^{2}\\
a^{2} & 0
\end{array}
\right)  .$%
\end{tabular}
\ \label{g1}%
\end{equation}

The ternary multiplication $\mu_{\mathsf{G}}^{\left[  \mathsf{3}\right]  }$ is
nonderived (any even product gives a diagonal matrix that is out of the set
$G=\mathfrak{adiag}$), ternary noncommutative and has the form%
\begin{equation}
\mu_{\mathsf{G}}^{\left[  \mathsf{3}\right]  }\left[  \mathsf{g}%
_{i},\mathsf{g}_{j},\mathsf{g}_{k}\right]  =g\left(  m_{1},n_{1}\right)
g\left(  m_{2},n_{2}\right)  g\left(  m_{3},n_{3}\right)  =g\left(
m_{1}+n_{2}+m_{3},n_{1}+m_{2}+n_{3}\right)  . \label{ggg}%
\end{equation}

The ternary identity $\mathsf{e}_{\mathsf{G}}$ (\ref{eg}) is defined by%
\begin{equation}
\mu_{\mathsf{G}}^{\left[  \mathsf{3}\right]  }\left[  \mathsf{e}_{\mathsf{G}%
},\mathsf{e}_{\mathsf{G}},\mathsf{g}\right]  =\mathsf{g}.
\end{equation}
Using (\ref{gi}) and (\ref{ggg}) we obtain the identity in matrix form%
\begin{equation}
\mathsf{e}_{\mathsf{G}}=\mathfrak{e}\left(  t\right)  =\left(
\begin{array}
[c]{cc}%
0 & a^{t}\\
a^{3-t} & 0
\end{array}
\right)  ,\ \ \ \ \ \ t=\mathbb{Z}\operatorname{mod}3.
\end{equation}
So the ternary group $\mathsf{G}^{\left[  \mathsf{3}\right]  }$ has $3$
identities%
\begin{equation}
\mathsf{e}_{\mathsf{G},1}=\mathfrak{e}\left(  0\right)  =\left(
\begin{array}
[c]{cc}%
0 & e\\
e & 0
\end{array}
\right)  =\mathsf{g}_{1},\ \mathsf{e}_{\mathsf{G},2}=\mathfrak{e}\left(
1\right)  =\left(
\begin{array}
[c]{cc}%
0 & a\\
a^{2} & 0
\end{array}
\right)  =\mathsf{g}_{9},\ \mathsf{e}_{\mathsf{G},3}=\mathfrak{e}\left(
2\right)  =\left(
\begin{array}
[c]{cc}%
0 & a^{2}\\
a & 0
\end{array}
\right)  =\mathsf{g}_{6}.\label{ei}%
\end{equation}
Each element in $\mathsf{G}^{\left[  \mathsf{3}\right]  }$ has its querelement
$\mathsf{\bar{g}}\left(  \mathsf{g}\right)  $ or $\bar{g}\left(  m,n\right)
=\overline{g\left(  m,n\right)  }$ defined by%
\begin{equation}
\mu_{\mathsf{G}}^{\left[  \mathsf{3}\right]  }\left[  \mathsf{g}%
,\mathsf{g},\mathsf{\bar{g}}\right]  =\mathsf{g},\ \ \ g\left(  m,n\right)
g\left(  m,n\right)  \bar{g}\left(  m,n\right)  =g\left(  m,n\right)  ,
\end{equation}
and in the matrix form%
\begin{equation}
\mathsf{\bar{g}}_{i}=\bar{g}\left(  m,n\right)  =\left(
\begin{array}
[c]{cc}%
0 & a^{3-n}\\
a^{3-m} & 0
\end{array}
\right)  ,\ \ \ \ m,n=\mathbb{Z}\operatorname{mod}3.\label{gq}%
\end{equation}

Note that each element of $\mathsf{G}^{\left[  \mathsf{3}\right]  }$ is
polyadic multiplicative idemponent (\ref{id}) because from (\ref{ggg}) it
follows%
\begin{equation}
\mathsf{g}^{\left\langle 3\right\rangle }=\left(  \mu_{\mathsf{G}}^{\left[
\mathsf{3}\right]  }\right)  ^{\circ3}\left[  \mathsf{g}^{7}\right]
=\mathsf{g},\ \ \ \ \ \ \forall\mathsf{g}\in G.
\end{equation}
So during \textquotedblleft polyadization\textquotedblright\ $\mathit{C}%
_{3}\rightarrow adiag\left(  \mathit{C}_{3},\mathit{C}_{3}\right)  $ the
binary cyclic group of order $3$  becomes the ternary group of  idempotents of
polyadic power $3$.

Thus, we have that%
\begin{equation}
\mathsf{G}^{\left[  \mathsf{3}\right]  }=\left\langle \left\{  \mathsf{g}%
_{i}\right\}  \mid\mu_{\mathsf{G}}^{\left[  \mathsf{3}\right]  }%
,\mathsf{e}_{\mathsf{G},i},\overline{\left(  \ \right)  }\right\rangle ,
\label{g3}%
\end{equation}
where the queroperation $\overline{\left(  \ \right)  }$ is defined in
(\ref{gq}), becomes the ternary group having $3$ identities (\ref{ei}).

The polyadic group $\left(  \mathbf{2,3}\right)  $-ring $\mathcal{R}^{\left[
2,3\right]  }\left[  \mathsf{G}^{\left[  \mathsf{3}\right]  }\right]
=\mathrm{R}^{\left[  \mathbf{2,3}\right]  }$, as the set $\mathfrak{R}$
consists of the finite formal sums (\ref{rg}), and the $\alpha$th element of
$\mathfrak{R}$ has the form
\begin{equation}
\mathrm{r}\left(  \alpha\right)  =\ \underset{i}{\mathbf{\Sigma}}%
\ r_{i}\left(  \alpha\right)  \bullet\mathsf{g}_{i},\ \ \ \ \ \ \ r_{i}\left(
\alpha\right)  \in R,\ \ \ \mathsf{g}_{i}\in G,\ \ \mathrm{r}\left(
\alpha\right)  \in\mathfrak{R}, \label{rag}%
\end{equation}
where $r_{i}\left(  \alpha\right)  \in j\mathbb{Z}$ and $\mathsf{g}_{i}$ are
defined in (\ref{gi}). Manifestly, we obtain%
\begin{equation}
\mathrm{r}\left(  \alpha\right)  =\ \underset{i}{\mathbf{\Sigma}%
}\ \mathfrak{j}k_{i}\left(  \alpha\right)  \bullet\left(
\begin{array}
[c]{cc}%
0 & a^{m_{i}}\\
a^{n_{i}} & 0
\end{array}
\right)  ,\ \ \ \ m_{i},n_{i}=\mathbb{Z}\operatorname{mod}3,\ \ k_{i}\left(
\alpha\right)  \in\mathbb{Z}. \label{ra1}%
\end{equation}
There is no identity in the polyadic group ring $\mathcal{R}^{\left[
2,3\right]  }\left[  \mathsf{G}^{\left[  \mathsf{3}\right]  }\right]  $,
because $\left(  2,3\right)  $-ring $\mathcal{R}^{\left[  2,3\right]  }$ is
unitless, but zero is%
\begin{equation}
\mathrm{z}_{\mathfrak{R}}=z_{R}\bullet\mathsf{g},\ \ \ \ \ \ \ \ z_{R}%
=\mathfrak{j}0=0\in\mathbb{Z},\ \ \ \mathsf{g}\in\mathsf{G}^{\left[
\mathsf{3}\right]  }.
\end{equation}

The commutative binary addition $\mathbf{\nu}_{\mathfrak{R}}^{\left[
\mathbf{2}\right]  }$ in $\mathrm{R}^{\left[  \mathbf{2,3}\right]  }$ is
simply \textquotedblleft left\textquotedblright-componentwise (\ref{mml}) and
reduces to the binary addition $\nu_{R}^{\left[  2\right]  }$ in
$\mathcal{R}^{\left[  2,3\right]  }$. The ternary multiplication $\mathbf{\mu
}_{\mathfrak{R}}^{\left[  \mathbf{3}\right]  }$ in the polyadic group $\left(
\mathbf{2,3}\right)  $-ring $\mathrm{R}^{\left[  2\mathbf{,3}\right]  }$ is
noncommutative and needs gathering of the similar coefficients before each
product of the ternary group $\mathsf{G}^{\left[  \mathsf{3}\right]  }$ (see
(\ref{ru}) and (\ref{mng})).

For instance, we have $3$ elements of the polyadic group $\left(
\mathbf{2,3}\right)  $-ring $\mathrm{R}^{\left[  \mathbf{2,3}\right]  }$%
\begin{align}
\mathrm{r}\left(  1\right)   &  =5\mathfrak{j}\bullet\mathsf{g}_{5}%
,\label{1}\\
\mathrm{r}\left(  2\right)   &  =2\mathfrak{j}\bullet\mathsf{g}_{7}%
\mathbf{\dotplus}\left(  -7\mathfrak{j}\right)  \bullet\mathsf{g}%
_{8},\label{2}\\
\mathrm{r}\left(  3\right)   &  =\left(  -4\mathfrak{j}\right)  \bullet
\mathsf{g}_{2}\dotplus7\mathfrak{j}\bullet\mathsf{g}_{3}\dotplus\left(
-3\mathfrak{j}\right)  \bullet\mathsf{g}_{6}, \label{3}%
\end{align}
where $\left(  \dotplus\right)  $ is the formal sum $\Sigma_{i}$.

Now we ternary multiply them (\ref{mmr}) and formally open the brackets (for
brevity, clearness and conciseness we do not write the ternary multiplications
manifestly)%
\begin{align}
&  \mathrm{r}_{0}=\mathbf{\mu}_{\mathfrak{R}}^{\left[  \mathbf{3}\right]
}\left[  \mathrm{r}\left(  1\right)  ,\mathrm{r}\left(  2\right)
,\mathrm{r}\left(  3\right)  \right]  =\left(  5\mathfrak{j}\bullet
\mathsf{g}_{5}\right)  \left(  2\mathfrak{j}\bullet\mathsf{g}_{7}%
\dotplus\left(  -7\mathfrak{j}\right)  \bullet\mathsf{g}_{8}\right)  \left(
\left(  -5\mathfrak{j}\right)  \bullet\mathsf{g}_{2}\dotplus7\mathfrak{j}%
\bullet\mathsf{g}_{3}\dotplus\left(  -3\mathfrak{j}\right)  \bullet
\mathsf{g}_{6}\right) \nonumber\\
&  =\left(  5\mathfrak{j}\bullet\mathsf{g}_{5}\right)  \left(  2\mathfrak{j}%
\bullet\mathsf{g}_{7}\right)  \left(  \left(  -4\mathfrak{j}\right)
\bullet\mathsf{g}_{2}\right)  \dotplus\left(  5\mathfrak{j}\bullet
\mathsf{g}_{5}\right)  \left(  3\mathfrak{j}\bullet\mathsf{g}_{7}\right)
\left(  7\mathfrak{j}\bullet\mathsf{g}_{3}\right)  \dotplus\left(
5\mathfrak{j}\bullet\mathsf{g}_{5}\right)  \left(  2\mathfrak{j}%
\bullet\mathsf{g}_{7}\right)  \left(  \left(  -3\mathfrak{j}\right)
\bullet\mathsf{g}_{6}\right) \nonumber\\
&  \dotplus\left(  5\mathfrak{j}\bullet\mathsf{g}_{5}\right)  \left(  \left(
-7\mathfrak{j}\right)  \bullet\mathsf{g}_{8}\right)  \left(  \left(
-4\mathfrak{j}\right)  \bullet\mathsf{g}_{2}\right)  \dotplus\left(
5\mathfrak{j}\bullet\mathsf{g}_{5}\right)  \left(  \left(  -7\mathfrak{j}%
\right)  \bullet\mathsf{g}_{8}\right)  \left(  7\mathfrak{j}\bullet
\mathsf{g}_{3}\right) \nonumber\\
&  \dotplus\left(  5\mathfrak{j}\bullet\mathsf{g}_{5}\right)  \left(  \left(
-7\mathfrak{j}\right)  \bullet\mathsf{g}_{8}\right)  \left(  \left(
-3\mathfrak{j}\right)  \bullet\mathsf{g}_{6}\right)  . \label{m3r}%
\end{align}
Then we use the ternary group multiplication (with manifest ternary
multiplication $\mu_{R}^{\left[  3\right]  }$ in the initial ring
$\mathcal{R}^{\left[  2,3\right]  }$ and $\mu_{\mathsf{G}}^{\left[
\mathsf{3}\right]  }$ in the ternary multiplication in $\mathsf{G}^{\left[
\mathsf{3}\right]  }$)%
\begin{align}
\mathrm{r}_{0}  &  =\mu_{R}^{\left[  3\right]  }\left[  5\mathfrak{j}%
,2\mathfrak{j},\left(  -4\mathfrak{j}\right)  \right]  \bullet\mu_{\mathsf{G}%
}^{\left[  \mathsf{3}\right]  }\left[  \mathsf{g}_{5},\mathsf{g}%
_{7},\mathsf{g}_{2}\right]  \dotplus\mu_{R}^{\left[  3\right]  }\left[
5\mathfrak{j},2\mathfrak{j},7\mathfrak{j}\right]  \bullet\mu_{\mathsf{G}%
}^{\left[  \mathsf{3}\right]  }\left[  \mathsf{g}_{5},\mathsf{g}%
_{7},\mathsf{g}_{3}\right] \nonumber\\
&  \dotplus\mu_{R}^{\left[  3\right]  }\left[  5\mathfrak{j},2\mathfrak{j}%
,\left(  -3\mathfrak{j}\right)  \right]  \bullet\mu_{\mathsf{G}}^{\left[
\mathsf{3}\right]  }\left[  \mathsf{g}_{5},\mathsf{g}_{7},\mathsf{g}%
_{6}\right]  \dotplus\mu_{R}^{\left[  3\right]  }\left[  5\mathfrak{j},\left(
-7\mathfrak{j}\right)  ,\left(  -4\mathfrak{j}\right)  \right]  \bullet
\mu_{\mathsf{G}}^{\left[  \mathsf{3}\right]  }\left[  \mathsf{g}%
_{5},\mathsf{g}_{8},\mathsf{g}_{2}\right] \nonumber\\
&  \dotplus\mu_{R}^{\left[  3\right]  }\left[  5\mathfrak{j},\left(
-7\mathfrak{j}\right)  ,7\mathfrak{j}\right]  \bullet\mu_{\mathsf{G}}^{\left[
\mathsf{3}\right]  }\left[  \mathsf{g}_{5},\mathsf{g}_{8},\mathsf{g}%
_{3}\right]  \dotplus\mu_{R}^{\left[  3\right]  }\left[  5\mathfrak{j},\left(
-7\mathfrak{j}\right)  ,\left(  -3\mathfrak{j}\right)  \right]  \bullet
\mu_{\mathsf{G}}^{\left[  \mathsf{3}\right]  }\left[  \mathsf{g}%
_{5},\mathsf{g}_{8},\mathsf{g}_{6}\right]  .
\end{align}

Performing the ternary multiplications $\mu_{R}^{\left[  3\right]  }$ and
$\mu_{\mathsf{G}}^{\left[  \mathsf{3}\right]  }$ in $\mathcal{R}^{\left[
5,3\right]  }$ we obtain the formal sum%
\begin{equation}
\mathrm{r}_{0}=40\mathfrak{j}\bullet\mathsf{g}_{5}\dotplus\left(
-70\mathfrak{j}\right)  \bullet\mathsf{g}_{6}\dotplus30\mathfrak{j}%
\bullet\mathsf{g}_{9}\dotplus\left(  -140\right)  \mathfrak{j}\bullet
\mathsf{g}_{9}\dotplus245\mathfrak{j}\bullet\mathsf{g}_{9}\dotplus\left(
-105\mathfrak{j}\right)  \bullet\mathsf{g}_{3}.\label{r00}%
\end{equation}

Finally, we gather coefficients from the initial ring $\mathcal{R}^{\left[
2,3\right]  }$ to get the ternary product of elements (\ref{1})--(\ref{3})
from the polyadic group ring $\mathrm{R}^{\left[  \mathbf{2,3}\right]  }$
(\ref{ra1})%
\begin{equation}
\mathrm{r}_{0}=\left(  -105\mathfrak{j}\right)  \bullet\mathsf{g}_{3}%
\dotplus40\mathfrak{j}\bullet\mathsf{g}_{5}\dotplus\left(  -70\mathfrak{j}%
\right)  \bullet\mathsf{g}_{6}\dotplus135\mathfrak{j}\bullet\mathsf{g}_{9}.
\label{r01}%
\end{equation}

The polyadic augmentation map $\varepsilon$ (\ref{aug}) for the elements
(\ref{1})--(\ref{3}) and (\ref{r01}) becomes%
\begin{align}
\varepsilon\left(  \mathrm{r}\left(  1\right)  \right)   &  =5\mathfrak{j},\\
\varepsilon\left(  \mathrm{r}\left(  2\right)  \right)   &  =\varepsilon
\left(  2\mathfrak{j}\bullet\mathsf{g}_{7}\mathbf{\dotplus}\left(
-7\mathfrak{j}\right)  \bullet\mathsf{g}_{8}\right)  =\nu_{R}^{\left[
2\right]  }\left[  2\mathfrak{j}\mathbf{,}\left(  -7\mathfrak{j}\right)
\right]  =-5\mathfrak{j},\\
\varepsilon\left(  \mathrm{r}\left(  3\right)  \right)   &  =\varepsilon
\left(  \left(  -4\mathfrak{j}\right)  \bullet\mathsf{g}_{2}\dotplus
7\mathfrak{j}\bullet\mathsf{g}_{3}\dotplus\left(  -3\mathfrak{j}\right)
\bullet\mathsf{g}_{6}\right)  =\left(  \nu_{R}^{\left[  2\right]  }\right)
^{\circ2}\left[  \left(  -4\mathfrak{j}\right)  ,7\mathfrak{j},\left(
-3\mathfrak{j}\right)  \right]  =0,\\
\varepsilon\left(  \mathrm{r}_{0}\right)   &  =\varepsilon\left(  \left(
-105\mathfrak{j}\right)  \bullet\mathsf{g}_{3}\dotplus40\mathfrak{j}%
\bullet\mathsf{g}_{5}\dotplus\left(  -70\mathfrak{j}\right)  \bullet
\mathsf{g}_{6}\dotplus135\mathfrak{j}\bullet\mathsf{g}_{9}\right) \nonumber\\
&  =\left(  \nu_{R}^{\left[  2\right]  }\right)  ^{\circ3}\left[  \left(
-105\mathfrak{j}\right)  ,40\mathfrak{j},\left(  -70\mathfrak{j}\right)
,135\mathfrak{j}\right]  =0,
\end{align}
where on the r.h.s. there are binary additions in $\mathcal{R}^{\left[
2,3\right]  }$, while on the l.h.s. there are formal sums. According to
(\ref{ke}) the elements $\mathrm{r}\left(  3\right)  $ and $\mathrm{r}_{0}$
are in the kernel.
\end{example}

Let us modify the previous construction for higher polyadic power case.

\begin{example}
Now we take for the initial ring the commutative nonderived $\left(
2,5\right)  $-ring with the underlying set $R=\mathfrak{j}_{4}\mathbb{Z}$
($\mathfrak{j}_{4}^{4}=-1$), operations are in $\mathbb{C}$. The operations
$\nu_{R}^{\left[  2\right]  }$ and $\mu_{R}^{\left[  5\right]  }$are usual
addition and product being the binary addition and nonderived $5$-ary
multiplication (only product of $5$ elements is closed). Thus $\left\langle
\mathfrak{j}_{4}\mathbb{Z}\mid\nu_{R}^{\left[  2\right]  }\right\rangle $ is a
binary group with respect to addition, and $\left\langle \mathfrak{j}%
_{4}\mathbb{Z}\mid\mu_{R}^{\left[  5\right]  }\right\rangle $ is not a $5$-ary
group, but only a $5$-ary semigroup, because there is no multiplicative
querelement for each $r\in R$ (no division in $\mathbb{Z}$). Obviously, that
$\mathcal{R}^{\left[  2,5\right]  }$ is unitless. Note that polyadic
distributivity follows from the binary distributivity in $\mathbb{Z}$.
Therefore,%
\begin{equation}
\mathcal{R}^{\left[  2,5\right]  }=\left\langle \mathfrak{j}_{4}\mathbb{Z}%
\mid\nu_{R}^{\left[  2\right]  },\mu_{R}^{\left[  5\right]  },z_{R}%
\right\rangle
\end{equation}
is a commutative nonderived $\left(  2,5\right)  $-ring without multiplicative
neutral element and $z_{R}=\mathfrak{j}0=0$.

The ternary group $\mathsf{G}^{\left[  \mathsf{3}\right]  }$ is the same as in
the previous example (\ref{g3}). The element of the polyadic group ring has
the same form (\ref{rag}), but now the multiplication (\ref{mmr}) is of $2$nd
polyadic power for $\mathsf{G}^{\left[  \mathsf{3}\right]  }$, i.e.
$\ell_{\mathsf{g}}=2$,%
\begin{align}
\left(  \mu_{\mathsf{G}}^{\left[  \mathsf{3}\right]  }\right)  ^{\circ
2}\left[  \mathsf{g}_{i_{1}},\mathsf{g}_{i_{2}},\mathsf{g}_{i_{3}}%
,\mathsf{g}_{i_{4}},\mathsf{g}_{i_{5}}\right]   &  =g\left(  m_{1}%
,n_{1}\right)  g\left(  m_{2},n_{2}\right)  g\left(  m_{3},n_{3}\right)
g\left(  m_{4},n_{4}\right)  g\left(  m_{5},n_{5}\right) \nonumber\\
&  =g\left(  m_{1}+n_{2}+m_{3}+n_{4}+m_{5},n_{1}+m_{2}+n_{3}+m_{4}%
+n_{5}\right)  . \label{g5}%
\end{align}

The initial ring $\mathcal{R}^{\left[  2,5\right]  }$ is still of $1$st
multiplicative polyadic power%
\begin{equation}
\mu_{R}^{\left[  5\right]  }\left[  \mathfrak{j}_{4}k_{1},\mathfrak{j}%
_{4}r_{2},\mathfrak{j}_{4}k_{3},\mathfrak{j}_{4}k_{4},\mathfrak{j}_{4}%
k_{5}\right]  =\left(  \mathfrak{j}_{4}k_{1}\right)  \left(  \mathfrak{j}%
_{4}r_{2}\right)  \left(  \mathfrak{j}_{4}k_{3}\right)  \left(  \mathfrak{j}%
_{4}k_{4}\right)  \left(  \mathfrak{j}_{4}k_{5}\right)  =-\mathfrak{j}%
_{4}k_{1}k_{2}k_{3}k_{4}k_{5},\ \ \ r_{i}\in\mathbb{Z}, \label{mk}%
\end{equation}
So manifestly the multiplicative arity of the polyadic group ring (\ref{nln})
using the \textquotedblleft quantization\textquotedblright\ condition
(\ref{lll}) with arities $n_{r}=5$ and $\mathsf{n}_{g}=3$, and different
polyadic powers $\ell_{n}=1$ and $\ell_{g}=2$ becomes%
\begin{equation}
\mathbf{n}_{r}=\ell_{n}\left(  n_{r}-1\right)  +1=\ell_{g}\left(
\mathsf{n}_{g}-1\right)  +1=1(5-1)+1=2(3-1)+1=5.
\end{equation}
The additive arity is inherited from the initial ring $\mathcal{R}^{\left[
2,5\right]  }$ and is binary $\mathbf{n}_{r}=n_{r}=2$. Thus, the polyadic
group ring is actually a nonderived $\left(  \mathbf{2,5}\right)  $-ring
$\mathrm{R}^{\left[  \mathbf{2,5}\right]  }$ with $2$nd polyadic power of the
ternary group $\mathsf{G}^{\left[  \mathsf{3}\right]  }$ multiplication. The
computations similar to the previous example can be done using the $5$-ary
multiplications (\ref{g5}) and (\ref{mk}), which manifestly shows the
existence of higher power polyadic group rings.
\end{example}

\section{\textsc{Conclusion}}

This article has laid the foundational framework for the theory of polyadic
group rings, a novel algebraic structure that generalizes the classical group
ring construction $\mathcal{R}[\mathsf{G}]$ to the higher arity setting. We
have formally defined the polyadic group ring $\mathrm{R}^{[\mathbf{m}%
_{r},\mathbf{n}_{r}]}=\mathcal{R}^{[m_{r},n_{r}]}[\mathsf{G}^{[n_{g}]}]$,
constructing its $\mathbf{m}_{r}$-ary addition and $\mathbf{n}_{r}$-ary
multiplication by systematically generalizing the binary operations. A central
achievement of this work is the derivation of the "quantization" conditions
that govern the admissible arities, revealing the profound interplay between
the arities of the initial ring, the initial polyadic group, and the resulting
polyadic group ring, including the novel case of operations with higher
polyadic powers.

We have established essential algebraic properties of these structures,
proving conditions for total associativity and characterizing the existence of
a zero element and identity. Furthermore, the generalization of key tools such
as the augmentation map and augmentation ideal provides a bridge for
transferring techniques from the classical theory into this new domain. The
explicit, non-trivial examples involving nonderived rings and finite polyadic
groups serve to concretely illustrate the theory and demonstrate the
noncommutative, convoluted nature of the multiplication.

This work opens several avenues for future research. The representation theory
of polyadic group rings, their homology, and other homological invariants
remain entirely unexplored and represent a natural next step. Furthermore, the
recent emergence of applied polyadic structures, such as in the
\textquotedblleft polyadic encryption\textquotedblright\ scheme \cite{dup/guo}%
, strongly validates the practical potential of this foundational work. The
complex, multi-operand relationships inherent in polyadic group rings make
them a promising candidate for constructing new cryptographic primitives,
developing non-linear codes, and modeling complex systems where binary
operations are insufficient. Thus, the theory of polyadic group rings not only
enriches pure algebra but also provides a new mathematical language for the
challenges of modern computation and security.

\bigskip

\textbf{Acknowledgements}. The author is grateful to Qiang Guo 
(College of Information and Communication, Harbin Engineering University) and 
Raimund Vogl (Center for Information Technology, University of M\"unster) for their invaluable support. The author also wishes to thank Hailong Chen 
and Yantai Research Institute (China) for gracious hospitality.

\newpage

\pagestyle{emptyf}
\mbox{}

\end{document}